\documentclass[final,hidelinks,onefignum,onetabnum]{siamart251216}

\usepackage{amsfonts}
\usepackage{pifont}
\usepackage{mathtools}
\usepackage{mathrsfs}
\usepackage{graphicx}
\usepackage{color}
\usepackage[dvipsnames]{xcolor}
\usepackage{tcolorbox}
\tcbuselibrary{breakable,skins}
\usepackage{ulem}
\usepackage{multicol}
\usepackage{cite} 
\usepackage{placeins}
\usepackage{tikz}
\usepackage{pgfplots}
\usetikzlibrary{shapes}
\usepackage{epstopdf}
\usepackage{algorithm}
\usepackage{algorithmic}

\usepackage{url}
\usepackage{hyperref}
\usepackage{float}
\usepackage{booktabs}
\usepackage{subcaption}
\usepackage{placeins}
\ifpdf
  \DeclareGraphicsExtensions{.eps,.pdf,.png,.jpg}
\else
  \DeclareGraphicsExtensions{.eps}
\fi

\newsiamremark{remark}{Remark}
\newsiamremark{hypothesis}{Hypothesis}
\crefname{hypothesis}{Hypothesis}{Hypotheses}
\newsiamthm{claim}{Claim}
\newsiamremark{fact}{Fact}
\crefname{fact}{Fact}{Facts}

\headers{SKETCH AND RESTART}{S.~G{\"{U}}TTEL, J.~LIU, AND L.~NYMAN}

\title{Sketch-and-Restart: Randomized Sketching in Quadrature-Based Restarting for Matrix Functions
}

\author{
Stefan G{\"{u}}ttel\thanks{Department of Mathematics, The University of Manchester, Oxford Road, Manchester, M139PL,
United Kingdom, \texttt{stefan.guettel@manchester.ac.uk}}
\and Jingyu Liu\thanks{School of Mathematical Sciences, Fudan University, 220 Handan Road, Shanghai, 200433, China, \texttt{jyliu22@m.fudan.edu.cn}}
\and Lauri Nyman\thanks{Department of Mathematics, The University of Manchester, Oxford Road, Manchester, M139PL,
United Kingdom, \texttt{lauri.nyman@manchester.ac.uk}}
}

\usepackage{amsopn}

\newcommand*{\trans}{\top}
\newcommand*{\herm}{*}
\newcommand*{\pinv}{\dagger}
\newcommand*{\real}{\mathbb{R}}
\newcommand*{\complex}{\mathbb{C}}

\newcommand*{\Mat}[1]{\boldsymbol{#1}}
\newcommand*{\tMat}[1]{\widetilde{\Mat{#1}}}

\newcommand*{\subspace}[1]{\mathcal{#1}}
\newcommand*{\bigO}{\mathscr{O}}

\newcommand*{\diff}{\,\mathrm{d}}
\newcommand*{\myspan}{\operatorname{span}}
\newcommand*{\errfunc}{\mathrm{err}}
\newcommand*{\range}{\mathcal{R}}
\newcommand*{\spec}{\varLambda}

\newcommand*{\sign}{\operatorname{sign}}
\newcommand*{\cond}{\kappa}
\newcommand*{\numrange}{\mathcal{W}}

\newcommand*{\iter}{\mathrm{iter}}
\newcommand*{\matvec}{\mathrm{matvec}}

\newcommand*{\actual}{\mathrm{act}}

\newcommand*{\Kry}{\mathcal{K}}

\newcommand*{\poly}{\mathcal{P}}
\newcommand*{\apply}{\mathrm{apply}}
\newcommand*{\sketching}{\mathrm{sketching}}

\newcommand*{\orthBasis}{\Mat{V}}

\newcommand*{\Basis}{\Mat{B}}
\newcommand*{\basis}{\Mat{b}}
\newcommand*{\updatedbasis}{\tMat{b}}

\newcommand*{\Hessen}{\Mat{H}}
\newcommand*{\uHessen}{\underline{\Mat{H}}}
\newcommand*{\hessen}{h}
\newcommand*{\updatedHessen}{\tMat{H}}
\newcommand*{\updatedhessen}{\tilde{h}}

\newcommand*{\tbasis}{\tMat{b}}

\newcommand*{\tskbasis}{\tMat{p}}
\newcommand*{\skbasis}{\Mat{p}}
\newcommand*{\skBasis}{\Mat{P}}

\newcommand*{\Abasis}{\Mat{z}}
\newcommand*{\ABasis}{\Mat{Z}}

\newcommand*{\skAbasis}{\Mat{q}}
\newcommand*{\skABasis}{\Mat{Q}}

\newcommand*{\RightHalfPlane}{\complex_{+}}

\newcommand*{\Arn}{\mathrm{Arn}}
\newcommand*{\HmArn}{\mathrm{HmArn}}
\newcommand*{\sArn}{\mathrm{sArn}}
\newcommand*{\sHmArn}{\mathrm{sHmArn}}

\newcommand*{\fail}{\mathrm{fail}}

\newcommand*{\quadature}{\mathrm{quad}}
\newcommand*{\condcontrol}{\mathrm{cond}}

\newcommand{\CONTSTATE}[1]{%
  \item[]\hspace{\algorithmicindent}#1%
}
\newcommand{\subroutine}[1]{\texttt{#1}}
\newcommand{\algorithmicfunction}{\textbf{Function:}}
\newcommand{\FUNCTION}[1]{\item[\algorithmicfunction]#1}

\newcommand{\sArnoldi}{\subroutine{sArnoldi}}
\newcommand{\sHmArnoldi}{\subroutine{sHmArnoldi}}
\newcommand{\tArnoldi}{\subroutine{tArnoldi\_adaptive}}

\newcommand*{\orthcondition}{\mathcal{C}}

\newcommand*{\myref}{\mathrm{ref}}

\hypersetup{
  hidelinks
}

\begin{document}

\maketitle

\begin{abstract}
We develop a sketch-and-restart framework for 
computing the action of a matrix function on a vector, 
\(f(\Mat{A})\Mat{b}\), 
where \(\Mat{A}\) is large, sparse, and non-Hermitian.
The framework combines quadrature-based restarting with 
Arnoldi-like decompositions generated by 
sketched or truncated Arnoldi processes.
Within this framework, we develop two classes of restarted algorithms.
The first uses a fixed Krylov subspace dimension 
and is based either on the sketched Arnoldi process 
or on a new sketched harmonic Arnoldi process 
proposed in this work.
The second class chooses the Krylov subspace dimension adaptively 
by running the truncated Arnoldi process 
until the condition number of the generated basis, 
estimated from its sketch, 
exceeds a prescribed threshold.
We also establish the convergence of the 
restarted sketched harmonic Arnoldi method 
for Stieltjes functions 
under the assumption that \(\Mat{A}\) is positive real.
Numerical experiments demonstrate the effectiveness 
of the proposed framework, 
including the computational savings achieved through sketching, 
the storage reduction enabled by adaptive truncation, 
and the acceleration obtained from thick restarting.
\end{abstract}


\section{Introduction}
\label{sec:introduction}

This paper considers the computation of \(f(\Mat{A})\Mat{b}\), 
where \(\Mat{A} \in \complex^{N \times N}\) is a matrix, 
\(\Mat{b} \in \complex^{N}\) is a vector, 
and \(f\) is a suitable scalar function.
Such computations arise in many areas of scientific computing.

If \(\Mat{A}\) is small and dense, 
one can first compute \(f(\Mat{A})\) explicitly 
using algorithms such as those described in~\cite{Higham_2008}, 
and then apply it to \(\Mat{b}\).
For large matrices, however, this approach is infeasible 
because of its computational and memory costs.
\emph{Krylov subspace methods} are therefore commonly used 
to approximate \(f(\Mat{A})\Mat{b}\), 
as they require only matrix-vector products with \(\Mat{A}\), 
which can often be computed efficiently 
when \(\Mat{A}\) is sparse.
Their main limitation is the need to 
store and orthogonalize the Krylov basis.
For non-Hermitian matrices, 
the Arnoldi process requires \(\bigO(N m^{2})\) operations 
to generate \(m\) orthonormal basis vectors and 
\(\bigO(N m)\) storage for the basis, 
and these costs can become substantial when \(m\) is large.

Many approaches have been proposed to address these limitations.
One important class is based on \emph{restarting}, 
in which only a fixed number of basis vectors 
need to be stored~\cite{Eiermann_Ernst_2006, Afanasjew_Eiermann_Ernst_Guettel_2008}.
The \emph{quadrature-based restarting} method of~\cite{Frommer_Guettel_Schweitzer_2014} 
uses an integral representation of \(f\), 
together with interpolation polynomials, 
to derive an integral representation of the error.
Approximating this error representation 
by quadrature allows the process to be restarted 
without storing all previously generated basis vectors.
However, restarted methods may converge slowly, 
and this issue can be alleviated by \emph{thick restarting} strategies~\cite{Morgan_2002,Eiermann_Ernst_Guettel_2011, Burke_Guettel_2024}.

Another class of approaches is based on 
\emph{randomized sketching}~\cite{Guettel_Schweitzer_2023, Cortinovis_Kressner_Nakatsukasa_2024}.
These methods approximate \(f(\Mat{A})\Mat{b}\) 
using a nonorthonormal basis of a Krylov subspace, 
often generated by 
the sketched Arnoldi process~\cite{Balabanov_Grigori_2022} 
or the truncated Arnoldi process~\cite{Nakatsukasa_Tropp_2024}.
The sketched Arnoldi process computes 
the inner products required for orthogonalization 
in the sketched space, 
thereby reducing the orthogonalization cost 
by nearly one half.
The truncated Arnoldi process 
reduces the orthogonalization cost to \(\bigO(N m t)\), 
where \(t\) is the truncation parameter.
However, the condition number of the generated basis 
may grow exponentially~\cite{Beckermann_Townsend_2017}, 
which limits the usable Krylov subspace dimension.

We focus on the case in which 
\(\Mat{A}\) is large, sparse, and non-Hermitian.
We develop a sketch-and-restart framework 
for computing \(f(\Mat{A}) \Mat{b}\) 
by combining randomized sketching with quadrature-based restarting.
The Krylov methods considered here, 
including their sketched variants, 
approximate \(f(\Mat{A}) \Mat{b}\) 
by \(p_{m - 1}(\Mat{A}) \Mat{b}\) 
in \(\myspan\{\Mat{b}, \Mat{A} \Mat{b}, \dotsc, \Mat{A}^{m - 1} \Mat{b}\}\).
The polynomial \(p_{m - 1}\) has degree at most \(m - 1\) 
and interpolates \(f\) 
at the eigenvalues of the \(m \times m\) reduced matrix 
obtained by an Arnoldi-type process.
This representation allows us to 
apply the existing integral error representation 
and quadrature-based restarting of~\cite{Frommer_Guettel_Schweitzer_2014},
with the corresponding projection conditions specified in Section~\ref{sec:krylov_skrylov}.

Within this framework, 
we develop two classes of restarted algorithms.
The first uses a fixed Krylov subspace dimension 
and is based either on the existing sketched Arnoldi process 
or on a new sketched harmonic Arnoldi process proposed in this work.
The sketched harmonic Arnoldi process 
generates a nonorthonormal Krylov basis 
satisfying the sketched harmonic orthogonality condition 
needed for the restarting framework.
These processes lead, respectively, to the restarted sketched
Arnoldi~(sArn) and restarted sketched harmonic Arnoldi~(sHmArn)
methods.
The second class comprises adaptive restarted Krylov methods.
The Krylov subspace dimension is selected 
by running the truncated Arnoldi process 
until the condition number of the generated basis, 
estimated from its sketch, 
exceeds a prescribed threshold.
A rank-\(1\) update enforces 
the Euclidean Arnoldi 
or harmonic Arnoldi orthogonality condition, 
yielding adaptive restarted Arnoldi (adaArn-\(t\)) 
or adaptive restarted harmonic Arnoldi (adaHmArn-\(t\)), respectively.
For Stieltjes functions, we relate the non-sketched adaptive variants 
to the existing convergence results of~\cite{FGS14b}
and derive conditional convergence bounds 
for sHmArn when \(\Mat{A}\) is positive real.

Recently, Guidotti et al.~\cite{Guidotti_Martinsson_Acebron_Monteiro_2025} 
combine randomized Arnoldi with restarting 
by evaluating the function of a block Hessenberg matrix 
that grows with the number of restart cycles.
We instead evaluate restart corrections by quadrature 
using an integral error representation 
on the current projected matrix.
For identical cycle decompositions and exact quadrature, 
their block formulation and our fixed-dimension restarted sArnoldi method 
yield the same approximants in exact arithmetic.

The remainder of this paper is organized as follows.
Section~\ref{sec:krylov_skrylov} formulates 
Krylov and sketched Krylov approximations 
within a common Arnoldi-like framework.
Section~\ref{sec:quad_restarting} extends 
quadrature-based restarting to this framework.
Section~\ref{sec:sketching_quad_restarting} 
presents the proposed algorithms, 
while Section~\ref{sec:convergence_analysis} analyzes their convergence.
Numerical experiments are presented in Section~\ref{sec:numerical_experiments}.
Finally, Section~\ref{sec:conclusions} concludes the paper.

We use the following notation.
The inner product of two vectors \(\Mat{x}, \Mat{y} \in \complex^{N}\) is defined as \(\langle \Mat{x}, \Mat{y} \rangle = \Mat{y}^{\herm} \Mat{x}\).
The open right half-plane is denoted by
\(\RightHalfPlane = \{z \in \complex : \Re(z) > 0\}\).
The symbol \(\|\cdot\|\) denotes 
the Euclidean norm for vectors 
and the spectral norm for matrices.
For a matrix \(\Mat{A}\), 
we denote its spectrum by \(\spec(\Mat{A})\), 
its condition number by
\(\cond(\Mat{A}) = \|\Mat{A}\| \|\Mat{A}^{-1}\|\), 
and its numerical range by
\begin{equation*}
\numrange(\Mat{A})
:=
\{\Mat{x}^{\herm}\Mat{A}\Mat{x} : \Mat{x} \in \complex^{N},\ \|\Mat{x}\| = 1\}.
\end{equation*}
A matrix \(\Mat{A} \in \complex^{N \times N}\) is called 
positive real if
\(\Re(\Mat{x}^{\herm}\Mat{A}\Mat{x}) > 0\) for every nonzero
\(\Mat{x} \in \complex^{N}\).
A matrix \(\Mat{V} \in \complex^{N \times m}\) is called 
orthonormal if its columns form an orthonormal set.
For a vector \(\Mat{x}\) and a subspace \(\subspace{V}\), 
we write
\(\Mat{x} \perp \subspace{V}\) 
if \(\Mat{x}\) is orthogonal to \(\subspace{V}\).
If \(\Mat{V}\) is a basis of \(\subspace{V}\), 
this condition is equivalent to
\(\Mat{V}^{\herm}\Mat{x} = \Mat{0}\), 
and we also write
\(\Mat{x} \perp \Mat{V}\).
The set of polynomials of degree at most \(m\) 
is denoted by \(\poly_{m}\).

\section{Krylov subspace methods and their sketched variants}
\label{sec:krylov_skrylov}

This section places several Krylov 
and sketched Krylov approximations 
within a common Arnoldi-like framework.
Specifically, we show that the Arnoldi 
and harmonic Arnoldi approximations, 
together with their sketched variants, 
can all be expressed in the form 
\(\Basis_{m} f(\Hessen_{m})\Mat{e}_{1}\beta\), 
where \(\Basis_{m}\) is a possibly nonorthonormal basis 
of a Krylov subspace 
and \(\Hessen_{m}\) is an upper Hessenberg matrix 
arising from an Arnoldi-like decomposition.
This observation forms the algebraic foundation 
for the quadrature-based restarting mechanism developed in Section~\ref{sec:quad_restarting}.

Suppose that \(f\) is analytic on the closure 
of a domain \(\Omega \subset \complex\) 
containing \(\spec(\Mat{A})\), 
and admits the integral representation
\begin{equation}
\label{eq:f_integral_representation}
f(z)
=
\int_{\Gamma} \frac{g(t)}{t - z} \diff t,
\qquad z \in \Omega,
\end{equation}
where \(\Gamma \subset \complex \setminus \Omega\) 
is a contour 
and \(g\) is a suitable function on \(\Gamma\).
Such a representation follows, 
for example, 
from the Cauchy integral formula, 
in which case \(\Gamma\) encloses \(\spec(\Mat{A})\), 
and also includes Stieltjes functions such as 
\(f(z) = z^{-\alpha}\) with \(0 < \alpha < 1\); see also~\cite{Frommer_Guettel_Schweitzer_2014}.
Then
\begin{equation}
\label{eq:fA_b_integral_representation}
f(\Mat{A})\Mat{b}
=
\int_{\Gamma} g(t)(t\Mat{I} - \Mat{A})^{-1}\Mat{b} \diff t
=
\int_{\Gamma} g(t)\Mat{x}(t) \diff t,
\end{equation}
where \(\Mat{x}(t)\) solves the shifted linear system
\begin{equation}
\label{eq:shifted_linear_system}
(t\Mat{I} - \Mat{A})\Mat{x}(t)
=
\Mat{b}.
\end{equation}

Let
\(\Kry_{m + 1}(\Mat{A}, \Mat{b})
=
\myspan\{\Mat{b}, \Mat{A}\Mat{b}, \dotsc, \Mat{A}^{m}\Mat{b}\}\)
be the Krylov subspace 
generated by \(\Mat{A}\) and \(\Mat{b}\).
Throughout this paper, 
we assume that the Krylov subspaces are nondegenerate, 
i.e.,
\(\dim \Kry_{j}(\Mat{A}, \Mat{b}) = j\) for \(1 \leq j \leq m + 1\).
An associated \emph{Arnoldi-like decomposition} is
\begin{equation} \label{eq:arnoldi_like_decomp}
\begin{aligned}
\Mat{A}\Basis_{m}
&=
\Basis_{m}\Hessen_{m}
+
\basis_{m + 1}\hessen_{m + 1, m}\Mat{e}_{m}^{\trans}
\\
&=
[\Basis_{m} \ \basis_{m + 1}]
\begin{bmatrix}
    \Hessen_{m}
    \\
    \hessen_{m + 1, m}\Mat{e}_{m}^{\trans}
\end{bmatrix}
\eqcolon
\Basis_{m + 1}\uHessen_{m},
\end{aligned}
\end{equation}
where, for \(1 \leq j \leq m + 1\),
\(\Basis_{j} = [\basis_{1} \ \cdots \ \basis_{j}]
\in \complex^{N \times j}\) is a basis of
\(\Kry_{j}(\Mat{A}, \Mat{b})\).
Consequently, the matrix \(\Hessen_{m} \in \complex^{m \times m}\)
is upper Hessenberg.
We assume that \(\Mat{b} = \basis_{1}\beta\)
for some \(\beta \in \complex\).
When \(\Basis_{m + 1}\) is orthonormal, 
the Arnoldi-like decomposition~\eqref{eq:arnoldi_like_decomp} 
reduces to the standard \emph{Arnoldi decomposition}, 
which can be computed by the Arnoldi process~\cite{Saad_2003}.
Throughout this paper, 
the data \(\Basis_{m}\), 
\(\Hessen_{m}\), 
\(\basis_{m + 1}\), 
\(\hessen_{m + 1, m}\), 
and \(\beta\) in~\eqref{eq:arnoldi_like_decomp} 
depend on the chosen method.

\subsection{Krylov subspace methods}
\label{subsec:krylov}

In Krylov subspace methods, 
the solution \(\Mat{x}(t)\) 
of the shifted linear system~\eqref{eq:shifted_linear_system} 
is approximated in \(\Kry_{m}(\Mat{A}, \Mat{b})\) 
by
\(\Mat{x}_{m}(t) = \Basis_{m}\Mat{y}_{m}(t)\), 
where \(\Mat{y}_{m}(t) \in \complex^{m}\).
The corresponding residual is
\(\Mat{r}_{m}(t) 
= 
\Mat{b} - (t\Mat{I} - \Mat{A})\Mat{x}_{m}(t)\).
The coefficient vector \(\Mat{y}_{m}(t)\) is determined 
by imposing a \emph{Petrov--Galerkin condition} 
on \(\Mat{r}_{m}(t)\).

If the projection condition is
\(\Mat{r}_{m}(t) \perp \Kry_{m}(\Mat{A}, \Mat{b})\), then
\(\Mat{x}_{m}^{\Arn}(t)\) is the Arnoldi approximation to the solution of the shifted linear system~\eqref{eq:shifted_linear_system}.
This is equivalent to solving the shifted linear system
using the full orthogonalization method~(FOM)~\cite{Saad_2003}.
Solving the resulting reduced system gives
\begin{equation*}
\Mat{y}_{m}^{\Arn}(t)
=
\bigl(t\Mat{I}_{m} - \Basis_{m}^{\pinv}\Mat{A}\Basis_{m}\bigr)^{-1}
\Mat{e}_{1}\beta.
\end{equation*}
Replacing \(\Mat{x}(t)\) in~\eqref{eq:fA_b_integral_representation} 
by \(\Mat{x}_{m}^{\Arn}(t)\), 
we obtain
\begin{equation}
\label{eq:arnoldi_approx}
\begin{aligned}
\Mat{f}_{m}^{\Arn}
&=
\int_{\Gamma} g(t)\Mat{x}_{m}^{\Arn}(t) \diff t
=
\Basis_{m}
\biggl(
\int_{\Gamma}
g(t)
\bigl(t\Mat{I}_{m} - \Basis_{m}^{\pinv}\Mat{A}\Basis_{m}\bigr)^{-1}
\diff t
\biggr)
\Mat{e}_{1}\beta
\\
&=
\Basis_{m}
f\bigl(\Basis_{m}^{\pinv}\Mat{A}\Basis_{m}\bigr)
\Mat{e}_{1}\beta,
\end{aligned}
\end{equation}
provided that the corresponding matrix function is well-defined.
The vector \(\Mat{f}_{m}^{\Arn}\) in~\eqref{eq:arnoldi_approx} 
is called the \emph{Arnoldi approximation} 
to \(f(\Mat{A})\Mat{b}\).
If \(\basis_{m + 1} \perp \Basis_{m}\) 
in the Arnoldi-like decomposition~\eqref{eq:arnoldi_like_decomp}, 
then
\begin{equation}
\label{eq:arnoldi_approx_last_orth}
\Basis_{m}^{\pinv}\Mat{A}\Basis_{m}
=
\Hessen_{m}
\quad \text{and} \quad
\Mat{f}_{m}^{\Arn}
=
\Basis_{m}f(\Hessen_{m})\Mat{e}_{1}\beta.
\end{equation}
Thus, under this orthogonality condition, 
the Arnoldi approximation can be computed directly 
from the upper Hessenberg matrix \(\Hessen_{m}\).
This condition holds, in particular, 
when \(\Basis_{m + 1}\) is generated 
by the standard Arnoldi process.

If the projection condition is
\(\Mat{r}_{m}(t) \perp \Mat{A}\Kry_{m}(\Mat{A}, \Mat{b})\), 
then
\(\Mat{x}_{m}^{\HmArn}(t)\) 
is the harmonic Arnoldi approximation to the solution of the shifted linear system~\eqref{eq:shifted_linear_system}.
By a similar argument,
\begin{equation*}
\Mat{y}_{m}^{\HmArn}(t)
=
\Bigl(
t\Mat{I}_{m}
-
\bigl((\Mat{A}\Basis_{m})^{\herm}\Basis_{m}\bigr)^{-1}
(\Mat{A}\Basis_{m})^{\herm}(\Mat{A}\Basis_{m})
\Bigr)^{-1}
\Mat{e}_{1}\beta.
\end{equation*}
The corresponding \emph{harmonic Arnoldi~(HmArnoldi) approximation} to \(f(\Mat{A})\Mat{b}\) is
\begin{equation}
\label{eq:hm_arnoldi_approx}
\Mat{f}_{m}^{\HmArn}
=
\Basis_{m}
f\Bigl(
\bigl((\Mat{A}\Basis_{m})^{\herm}\Basis_{m}\bigr)^{-1}
(\Mat{A}\Basis_{m})^{\herm}(\Mat{A}\Basis_{m})
\Bigr)
\Mat{e}_{1}\beta,
\end{equation}
provided that the corresponding matrix function is well-defined.
If \(\basis_{m + 1} \perp \Mat{A}\Basis_{m}\) 
in the Arnoldi-like decomposition~\eqref{eq:arnoldi_like_decomp}, 
then
\begin{equation}
\label{eq:hm_arnoldi_approx_last_hm_orth}
\bigl((\Mat{A}\Basis_{m})^{\herm}\Basis_{m}\bigr)^{-1}
(\Mat{A}\Basis_{m})^{\herm}(\Mat{A}\Basis_{m})
=
\Hessen_{m}
\quad \text{and} \quad
\Mat{f}_{m}^{\HmArn}
=
\Basis_{m}f(\Hessen_{m})\Mat{e}_{1}\beta.
\end{equation}
Thus, under this orthogonality condition, the harmonic Arnoldi approximation has the same form as the Arnoldi approximation in~\eqref{eq:arnoldi_approx_last_orth}.
For \(t = 0\), \(\Mat{x}_{m}^{\HmArn}(0)\) coincides with
the generalized minimal residual~(GMRES) approximation
to the solution of~\eqref{eq:shifted_linear_system}
with a zero initial guess~\cite{Saad_2003}.
For general shifts \(t\), harmonic Arnoldi retains
the same test space
\(\Mat{A}\Kry_{m}(\Mat{A}, \Mat{b})\).
This connection to GMRES underlies the convergence analysis
of restarted harmonic Arnoldi for Stieltjes functions
of positive real matrices~\cite[Section~6]{FGS14b}.
These convergence results motivate our consideration
of harmonic Arnoldi and its sketched extension.

\subsection{Sketched Krylov subspace methods}
\label{subsec:skrylov}

In the Krylov subspace methods 
introduced in Section~\ref{subsec:krylov}, 
the shifted linear system~\eqref{eq:shifted_linear_system} 
is approximated by imposing a Petrov--Galerkin condition 
on the residual.
This procedure can be accelerated 
using the so-called \emph{sketch-and-solve} approach~\cite{Woodruff_2014}.
More specifically, suppose that, 
for a fixed constant \(\varepsilon \in [0,1)\), 
the matrix \(\Mat{S} \in \complex^{s \times N}\) 
satisfies the following \emph{subspace embedding property}:
\begin{equation}
\label{eq:subspace_embedding_property}
(1 - \varepsilon)\|\Mat{v}\|^{2}
\leq
\|\Mat{S}\Mat{v}\|^{2}
\leq
(1 + \varepsilon)\|\Mat{v}\|^{2},
\qquad
\forall \Mat{v} \in \Kry_{m + 1}(\Mat{A}, \Mat{b}).
\end{equation}
Condition~\eqref{eq:subspace_embedding_property} admits the following equivalent bound for the inner product:
\begin{equation}
\label{eq:subspace_embedding_property_inner_product}
\bigl|
\langle \Mat{u}, \Mat{v} \rangle
-
\langle \Mat{S}\Mat{u}, \Mat{S}\Mat{v} \rangle
\bigr|
\leq
\varepsilon\|\Mat{u}\|\|\Mat{v}\|,
\qquad
\forall \Mat{u}, \Mat{v}
\in \Kry_{m + 1}(\Mat{A}, \Mat{b}).
\end{equation}
For completeness, we briefly verify this equivalence
for complex vectors with the same \(\varepsilon\).

Let \(\orthBasis\) be an orthonormal basis 
of \(\Kry_{m + 1}(\Mat{A}, \Mat{b})\), 
and define \(\Mat{E} = \orthBasis^{\herm}
(\Mat{I} - \Mat{S}^{\herm}\Mat{S}) \orthBasis\).
Since \(\Mat{E}\) is Hermitian,~\eqref{eq:subspace_embedding_property} gives
\begin{equation*}
\|\Mat{E}\|
=
\max_{\|\Mat{z}\| = 1}
\bigl|\Mat{z}^{\herm}\Mat{E}\Mat{z}\bigr|
\leq \varepsilon.
\end{equation*}
Writing \(\Mat{u} = \orthBasis\Mat{x}\) and \(\Mat{v} = \orthBasis\Mat{y}\), we obtain
\begin{equation*}
\bigl|
\langle \Mat{u}, \Mat{v} \rangle
-
\langle \Mat{S}\Mat{u}, \Mat{S}\Mat{v} \rangle
\bigr|
=
\bigl|\Mat{x}^{\herm}\Mat{E}\Mat{y}\bigr|
\leq
\|\Mat{E}\|\|\Mat{x}\|\|\Mat{y}\|
\leq
\varepsilon\|\Mat{u}\|\|\Mat{v}\|.
\end{equation*}
Conversely, taking \(\Mat{u} = \Mat{v}\) in \eqref{eq:subspace_embedding_property_inner_product} recovers~\eqref{eq:subspace_embedding_property}.

We call such a matrix \(\Mat{S}\) a \emph{sketching matrix}, 
and \(s\) the sketching dimension.
Two vectors \(\Mat{x}, \Mat{y} \in \complex^{N}\) 
are called \(\Mat{S}\)-orthogonal, 
denoted by \(\Mat{x} \perp_{\Mat{S}} \Mat{y}\), 
if \(\Mat{S}\Mat{x} \perp \Mat{S}\Mat{y}\).
Typically, \(\Mat{S}\) is drawn from an
\emph{oblivious subspace embedding}~(OSE) family.
For \(\varepsilon, \delta_{\fail} \in (0,1)\),
an \((\varepsilon,\delta_{\fail})\) OSE family
is a distribution over matrices
\(\Mat{S} \in \complex^{s \times N}\)
that is independent of the target subspace
and has the following property:
for every fixed subspace
\(\mathcal{U} \subseteq \complex^{N}\)
with \(\dim \mathcal{U} \leq m+1\),
the norm inequalities in~\eqref{eq:subspace_embedding_property}
hold simultaneously for all \(\Mat{v} \in \mathcal{U}\)
with probability at least
\(1-\delta_{\fail}\)~\cite{Woodruff_2014}.
For example, consider sparse sign matrices with independent columns,
\begin{equation*}
\Mat{S}
=
\frac{1}{\sqrt{\zeta}}
\begin{bmatrix}
\Mat{s}_{1} & \cdots & \Mat{s}_{N}
\end{bmatrix}
\in \real^{s \times N}.
\end{equation*}
The columns are drawn independently.
Each column \(\Mat{s}_{i}\) has exactly \(\zeta\) nonzero entries,
whose positions are selected uniformly without replacement
from the \(s\) available positions.
The nonzero entries are assigned independent values
\(\pm 1\) with equal probability.
For any \(B > 2\), this family provides the above guarantee
whenever
\begin{equation*}
s \geq
C_1 \frac{B(m+1)\log((m+1)/\delta_{\fail})}{\varepsilon^2},
\qquad
\zeta \geq
C_2 \frac{\log((m+1)/\delta_{\fail})}{\varepsilon\log B},
\end{equation*}
where \(C_1,C_2>0\) are absolute constants
and \(1 \leq \zeta \leq s\)~\cite{Cohen_Nelson_Woodruff_2016}.

We first consider the sketched version of the Arnoldi approximation.
Let \(\Mat{x}_{m}^{\sArn}(t)\)
approximate the solution \(\Mat{x}(t)\) 
of~\eqref{eq:shifted_linear_system}.
The Petrov--Galerkin condition is imposed after sketching, namely,
\(\Mat{r}_{m} \perp_{\Mat{S}} \Kry_{m}(\Mat{A}, \Mat{b})\).
By the subspace embedding property~\eqref{eq:subspace_embedding_property}, 
\(\Mat{S}\Basis_{m}\) has full column rank.
Following the same argument 
as in Section~\ref{subsec:krylov}, 
we obtain
\begin{equation*}
\Mat{y}_{m}^{\sArn}(t)
=
\bigl(
t\Mat{I}_{m}
-
(\Mat{S}\Basis_{m})^{\pinv}
(\Mat{S}\Mat{A}\Basis_{m})
\bigr)^{-1}
\Mat{e}_{1}\beta
\end{equation*}
and
\begin{equation}
\label{eq:s_arnoldi_approx}
\Mat{f}_{m}^{\sArn}
=
\Basis_{m}
f\bigl(
(\Mat{S}\Basis_{m})^{\pinv}
(\Mat{S}\Mat{A}\Basis_{m})
\bigr)
\Mat{e}_{1}\beta,
\end{equation}
provided that the corresponding matrix function is well-defined.
The vector \(\Mat{f}_{m}^{\sArn}\) 
in~\eqref{eq:s_arnoldi_approx} 
is called the \emph{sketched Arnoldi~(sArnoldi) approximation} 
to \(f(\Mat{A})\Mat{b}\).
We point out that 
the sArnoldi approximation~\eqref{eq:s_arnoldi_approx} used here 
differs in form from that in~\cite{Guettel_Schweitzer_2023}, 
although the two formulations are mathematically equivalent.
Moreover, if \(\basis_{m + 1} \perp_{\Mat{S}} \Basis_{m}\) in the Arnoldi-like decomposition~\eqref{eq:arnoldi_like_decomp}, then
\begin{equation}
\label{eq:s_arnoldi_approx_last_sorth}
(\Mat{S}\Basis_{m})^{\pinv}
(\Mat{S}\Mat{A}\Basis_{m})
=
\Hessen_{m}
\quad \text{and} \quad
\Mat{f}_{m}^{\sArn}
=
\Basis_{m}f(\Hessen_{m})\Mat{e}_{1}\beta.
\end{equation}

A similar discussion applies to the HmArnoldi approximation.
By imposing the condition
\(\Mat{r}_{m} \perp_{\Mat{S}} \Mat{A}\Kry_{m}(\Mat{A}, \Mat{b})\), we obtain
\begin{equation*}
\Mat{y}_{m}^{\sHmArn}(t)
=
\Bigl(
t\Mat{I}_{m}
-
\bigl(
(\Mat{S}\Mat{A}\Basis_{m})^{\herm}
(\Mat{S}\Basis_{m})
\bigr)^{-1}
(\Mat{S}\Mat{A}\Basis_{m})^{\herm}
(\Mat{S}\Mat{A}\Basis_{m})
\Bigr)^{-1}
\Mat{e}_{1}\beta.
\end{equation*}
Consequently, 
\(\Mat{x}_{m}^{\sHmArn}(t) = \Basis_{m}\Mat{y}_{m}^{\sHmArn}(t)\) 
is the sketched harmonic Arnoldi approximation 
to the solution of the shifted linear system~\eqref{eq:shifted_linear_system}.
The corresponding \emph{sketched harmonic Arnoldi~(sHmArnoldi) approximation} 
to \(f(\Mat{A})\Mat{b}\) is
\begin{equation}
\label{eq:s_hm_arnoldi_approx}
\Mat{f}_{m}^{\sHmArn}
=
\Basis_{m}
f\Bigl(
\bigl(
(\Mat{S}\Mat{A}\Basis_{m})^{\herm}
(\Mat{S}\Basis_{m})
\bigr)^{-1}
(\Mat{S}\Mat{A}\Basis_{m})^{\herm}
(\Mat{S}\Mat{A}\Basis_{m})
\Bigr)
\Mat{e}_{1}\beta,
\end{equation}
provided that the corresponding matrix function is well-defined.
If \(\basis_{m + 1} \perp_{\Mat{S}} \Mat{A}\Basis_{m}\) 
in the Arnoldi-like decomposition~\eqref{eq:arnoldi_like_decomp}, 
then
\begin{equation}
\label{eq:s_hm_arnoldi_approx_last_shm_orth}
\begin{gathered}
\bigl(
(\Mat{S}\Mat{A}\Basis_{m})^{\herm}
(\Mat{S}\Basis_{m})
\bigr)^{-1}
(\Mat{S}\Mat{A}\Basis_{m})^{\herm}
(\Mat{S}\Mat{A}\Basis_{m})
=
\Hessen_{m},
\\
\Mat{f}_{m}^{\sHmArn}
=
\Basis_{m}f(\Hessen_{m})\Mat{e}_{1}\beta.
\end{gathered}
\end{equation}

\subsection{\texorpdfstring{Rank-\(1\) update of the Arnoldi-like decomposition}{Rank-1 update of the Arnoldi-like decomposition}}
\label{subsec:rank1_update}

The discussions in Sections~\ref{subsec:krylov} 
and~\ref{subsec:skrylov} show that 
enforcing a suitable orthogonality condition 
on the last vector \(\basis_{m + 1}\) 
can simplify the approximation formula 
by reducing it to a matrix function 
of the upper Hessenberg matrix \(\Hessen_{m}\).
This condition is not always satisfied 
by a given Arnoldi-like decomposition.
It can, however, be enforced 
by applying a rank-\(1\) update 
to~\eqref{eq:arnoldi_like_decomp}, 
as described below; 
see also~\cite{Cortinovis_Kressner_Nakatsukasa_2024, Grigori_Kressner_Shao_Simunec_2026} 
for related approaches.
The update below is a specialization 
of Stewart's translation operation 
for Krylov decompositions~\cite{Stewart_2002}.

Given~\eqref{eq:arnoldi_like_decomp}, 
we update \(\basis_{m + 1}\) to \(\updatedbasis_{m + 1}\) 
so that \(\updatedbasis_{m + 1}\) 
satisfies a prescribed orthogonality condition.
To this end, write
\(\basis_{m + 1}
=
\Basis_{m}\Mat{c}_{m}
+
\updatedbasis_{m + 1}\alpha_{m + 1}\),
where \(\Mat{c}_{m}\) is to be determined 
and \(\alpha_{m + 1}\) is chosen according to a normalization condition.
Substituting this expression into~\eqref{eq:arnoldi_like_decomp} gives
\begin{equation*}
\begin{aligned}
\Mat{A}\Basis_{m}
&=
\Basis_{m}\Hessen_{m}
+
\bigl(
    \Basis_{m}\Mat{c}_{m}
    +
    \updatedbasis_{m + 1}\alpha_{m + 1}
\bigr)
\hessen_{m + 1, m}\Mat{e}_{m}^{\trans}
\\
&=
\Basis_{m}
\bigl(
    \Hessen_{m}
    +
    \Mat{c}_{m}\hessen_{m + 1, m}\Mat{e}_{m}^{\trans}
\bigr)
+
\updatedbasis_{m + 1}
\bigl(
    \alpha_{m + 1}\hessen_{m + 1, m}
\bigr)
\Mat{e}_{m}^{\trans}
\\
&=
\Basis_{m}\updatedHessen_{m}
+
\updatedbasis_{m + 1}
\updatedhessen_{m + 1, m}
\Mat{e}_{m}^{\trans},
\end{aligned}
\end{equation*}
where
\(\updatedHessen_{m}
=
\Hessen_{m}
+
\Mat{c}_{m}\hessen_{m + 1, m}\Mat{e}_{m}^{\trans}\)
and
\(\updatedhessen_{m + 1, m}
=
\alpha_{m + 1}\hessen_{m + 1, m}\).
The matrix \(\updatedHessen_{m}\) remains upper Hessenberg.

If the desired orthogonality condition is
\(\updatedbasis_{m + 1} \perp \Basis_{m}\) or
\(\updatedbasis_{m + 1} \perp_{\Mat{S}} \Basis_{m}\), 
then \(\Mat{c}_{m} = \Basis_{m}^{\pinv}\basis_{m + 1}\) 
or \(\Mat{c}_{m} = (\Mat{S}\Basis_{m})^{\pinv}(\Mat{S}\basis_{m + 1})\),
respectively.
Similarly, if the desired orthogonality condition is
\(\updatedbasis_{m + 1} \perp \Mat{A}\Basis_{m}\) or
\(\updatedbasis_{m + 1} \perp_{\Mat{S}} \Mat{A}\Basis_{m}\), then
\begin{eqnarray*}
\Mat{c}_{m}
&=&
\bigl(
    (\Mat{A}\Basis_{m})^{\herm}\Basis_{m}
\bigr)^{-1}
(\Mat{A}\Basis_{m})^{\herm}\basis_{m + 1}\\
\text{or}\quad  \Mat{c}_{m}
&=&
\bigl(
    (\Mat{S}\Mat{A}\Basis_{m})^{\herm}
    (\Mat{S}\Basis_{m})
\bigr)^{-1}
(\Mat{S}\Mat{A}\Basis_{m})^{\herm}
(\Mat{S}\basis_{m + 1}),
\end{eqnarray*}
respectively.
The Euclidean orthogonality condition
\(\updatedbasis_{m + 1} \perp \Basis_{m}\) and the corresponding
rank-\(1\) correction are also used in the similarity-restoring
randomized Arnoldi process
of~\cite{Grigori_Kressner_Shao_Simunec_2026}.

Under the respective orthogonality conditions 
\(\basis_{m + 1} \perp \Basis_{m}\), 
\(\basis_{m + 1} \perp \Mat{A}\Basis_{m}\), 
\(\basis_{m + 1} \perp_{\Mat{S}} \Basis_{m}\), 
and \(\basis_{m + 1} \perp_{\Mat{S}} \Mat{A}\Basis_{m}\), 
the Arnoldi, HmArnoldi, sArnoldi, and sHmArnoldi approximations take the common \emph{Arnoldi-like approximation} form
\begin{equation}
\label{eq:arnoldi_like_approx}
\Mat{f}_{m}
\coloneq
\Basis_{m}f(\Hessen_{m})\Mat{e}_{1}\beta.
\end{equation}
Here, \(\Basis_{m}\), \(\Hessen_{m}\), and \(\beta\) are method-dependent.

\section{Restarting based on quadrature}
\label{sec:quad_restarting}

This section formulates 
the integral error representation 
and quadrature-based restarting of~\cite{Frommer_Guettel_Schweitzer_2014} 
for the four approximations in Section~\ref{sec:krylov_skrylov}, retaining the normalization factors for each cycle.
Our algorithmic contributions 
in Section~\ref{sec:sketching_quad_restarting} 
combine this mechanism with sketched basis generation, 
including the sketched Arnoldi process,
the new sketched harmonic Arnoldi process, 
and adaptive truncated Arnoldi processes.
The following lemma 
recalls the required integral representation 
of the interpolating polynomial~\cite[Lemma~3.1]{Frommer_Guettel_Schweitzer_2014}.

\begin{lemma}[Integral representation of the interpolating polynomial]
\label{lem:interp_poly_integral}
Let \(\Omega \subset \complex\) be a domain, and let \(f\colon \Omega \to \complex\) be analytic on the closure of \(\Omega\) and admit the integral representation~\eqref{eq:f_integral_representation}.
Then the polynomial \(p_{m - 1} \in \poly_{m - 1}\) interpolating \(f\) at the points \(\{\theta_{1}, \dotsc, \theta_{m}\} \subset \Omega\) is given by
\begin{equation}
\label{eq:interpolating_poly_integral}
p_{m - 1}(z)
=
\int_{\Gamma}
\biggl(
    1 - \frac{\phi_{m}(z)}{\phi_{m}(t)}
\biggr)
\frac{g(t)}{t - z}
\diff t,
\end{equation}
where \(\phi_{m}(z) = (z - \theta_{1})\dotsm(z - \theta_{m})\), provided that the integral in~\eqref{eq:interpolating_poly_integral} exists.
\end{lemma}

\subsection{Restarting the Arnoldi-like approximation}
\label{subsec:restart_arnoldi_like_approx}

We first recall the following polynomial identity for Arnoldi-like decompositions~\cite[Lemma~2.1]{Paige_Parlett_van_der_Vorst_1995}.

\begin{lemma}[Polynomial action in an Arnoldi-like decomposition]
\label{lem:poly_A_b1_arnoldi_like}
    Given an Arnoldi-like decomposition~\eqref{eq:arnoldi_like_decomp}, let
    \(p(z) = c_{0} + c_{1}z + \dotsb + c_{m}z^{m} \in \poly_{m}\).
    Then
    \begin{equation*}
        p(\Mat{A})\basis_{1}
        =
        \Basis_{m}p(\Hessen_{m})\Mat{e}_{1}
        +
        \basis_{m + 1}\gamma_{m}c_{m},
    \end{equation*}
    where \(\gamma_{m} = \prod_{j = 1}^{m}\hessen_{j + 1, j}\).
\end{lemma}

For the Arnoldi-like decomposition~\eqref{eq:arnoldi_like_decomp},
consider its corresponding Arnoldi-like approximation \(\Mat{f}_{m}\) in~\eqref{eq:arnoldi_like_approx}.
Assume that \(\spec(\Mat{A}) \cup \spec(\Hessen_{m}) \subset \Omega\), 
and that \(f\) is analytic on the closure of \(\Omega\) 
and admits~\eqref{eq:f_integral_representation}.
Let \(p_{m - 1} \in \poly_{m - 1}\) interpolate \(f\) at \(\spec( \Hessen_{m})\).
Then \(f(\Hessen_{m}) = p_{m - 1}(\Hessen_{m})\).
Since its degree is at most \(m - 1\), 
\(c_{m} = 0\) in Lemma~\ref{lem:poly_A_b1_arnoldi_like}; multiplying that identity by \(\beta\) gives
\begin{equation*}
\Mat{f}_{m}
=
\Basis_{m}p_{m - 1}(\Hessen_{m})\Mat{e}_{1}\beta
=
p_{m - 1}(\Mat{A})\basis_{1}\beta
=
p_{m - 1}(\Mat{A})\Mat{b},
\end{equation*}
where we used \(\Mat{b} = \basis_{1}\beta\).
Write
\(\spec(\Hessen_{m}) = \{\theta_{1},\dotsc,\theta_{m}\}\)
and define
\(\phi_{m}(z) = (z - \theta_{1})\dotsm(z - \theta_{m})\).
Then Lemma~\ref{lem:interp_poly_integral} gives
\begin{equation*}
f(z) - p_{m - 1}(z)
=
\int_{\Gamma}
\frac{\phi_{m}(z)}{\phi_{m}(t)}
\frac{g(t)}{t - z}
\diff t,
\end{equation*}
and therefore
\begin{equation*}
f(\Mat{A})\Mat{b} - \Mat{f}_{m}
=
\biggl(
\int_{\Gamma}
\frac{g(t)}{\phi_{m}(t)}
(t\Mat{I} - \Mat{A})^{-1}
\diff t
\biggr)
\phi_{m}(\Mat{A})\Mat{b}.
\end{equation*}
Moreover, Lemma~\ref{lem:poly_A_b1_arnoldi_like} gives
\begin{equation*}
\phi_{m}(\Mat{A})\Mat{b}
=
\Basis_{m}\phi_{m}(\Hessen_{m})\Mat{e}_{1}\beta
+
\basis_{m + 1}\gamma_{m}\beta
=
\basis_{m + 1}\gamma_{m}\beta,
\end{equation*}
where we used \(\phi_{m}(\Hessen_{m}) = \Mat{0}\).
Thus,
\begin{equation*}
    f(\Mat{A})\Mat{b} - \Mat{f}_{m}
    =
    \gamma_{m}\beta
    \biggl(
        \int_{\Gamma}
        \frac{g(t)}{\phi_{m}(t)}
        (t\Mat{I} - \Mat{A})^{-1}
        \diff t
    \biggr)
    \basis_{m + 1}.
\end{equation*}
The following theorem states the integral error formula
of~\cite[Theorem~3.4]{Frommer_Guettel_Schweitzer_2014}
in the Arnoldi-like formulation used here, 
with the normalization factor \(\beta\) retained explicitly.

\begin{theorem}[Error of the Arnoldi-like approximation]
\label{thm:err_arnoldi_like_approx}
Consider the Arnoldi-like decomposition~\eqref{eq:arnoldi_like_decomp} with \(\Mat{b} = \basis_{1}\beta\).
Suppose that \(f\) admits the integral representation~\eqref{eq:f_integral_representation}, and let \(\Mat{f}_{m}\) be the Arnoldi-like approximation defined in~\eqref{eq:arnoldi_like_approx}.
Define \(\phi_{m}(t) = (t - \theta_{1})\dotsm(t - \theta_{m})\), where \(\spec(\Hessen_{m}) = \{\theta_{1}, \dotsc, \theta_{m}\} \subset \Omega\), and let \(\gamma_{m} = \prod_{j = 1}^{m}\hessen_{j + 1, j}\).
Then
\begin{equation}
\label{eq:err_arnoldi_like_approx}
f(\Mat{A})\Mat{b} - \Mat{f}_{m}
=
\gamma_{m}\beta
\biggl(
\int_{\Gamma}
\frac{g(t)}{\phi_{m}(t)}
(t\Mat{I} - \Mat{A})^{-1}
\diff t
\biggr)
\basis_{m + 1}
\eqcolon
\errfunc(\Mat{A})\basis_{m + 1},
\end{equation}
provided that the integral exists.
Thus, the error function is
\begin{equation*}
\errfunc(z)
=
\gamma_{m}\beta
\int_{\Gamma}
\frac{g(t)}{\phi_{m}(t)}
(t - z)^{-1}
\diff t.
\end{equation*}
\end{theorem}

The choice of \(\basis_{m + 1}\) as the restart vector 
follows directly from the error representation~\eqref{eq:err_arnoldi_like_approx}.
Using \(\basis_{m + 1}\) as the new starting vector, 
suppose that a new Arnoldi-like decomposition 
\(\Mat{A}\Basis_{m}^{+} = \Basis_{m + 1}^{+}\uHessen_{m}^{+}\) 
is generated with \(\basis_{m + 1} = \basis_{1}^{+}\beta^{+}\).
Applying Theorem~\ref{thm:err_arnoldi_like_approx} to \(\errfunc\) 
and the new decomposition under its stated assumptions gives
\begin{equation*}
\errfunc(\Mat{A})\basis_{m + 1}
=
\Basis_{m}^{+}\errfunc(\Hessen_{m}^{+})
\Mat{e}_{1}\beta^{+}
+
\errfunc^{+}(\Mat{A})\basis_{m + 1}^{+},
\end{equation*}
where \(\errfunc^{+}(z)\) is the new error function.
The first term is the Arnoldi-like approximation and the second its exact remaining error.

Algorithm~\ref{alg:restart_arnoldi_like_framework} repeats this process.
For \(k \geq 1\), it determines \(\errfunc^{[k]}\) 
from the preceding decompositions 
using Corollary~\ref{cor:err_arnoldi_like_approx_k_restart}.
The new decomposition starts 
from \(\basis_{m + 1}^{[k - 1]} = \basis_{1}^{[k]}\beta^{[k]}\).
Adaptive quadrature evaluates 
\(\errfunc^{[k]}(\Hessen_{m}^{[k]}) \Mat{e}_{1}\beta^{[k]}\) 
using shifted systems 
with the current reduced matrix.
Multiplying the result by \(\Basis_{m}^{[k]}\) 
gives the correction \(\Mat{g}^{[k]}\), 
which is added to the current approximation.
Although the corollary 
assumes a fixed Krylov subspace dimension \(m\)
in each restart cycle, 
the construction extends directly to variable dimensions.

\begin{algorithm}[!htbp]
\caption{Restarting framework for computing \(f(\Mat{A})\Mat{b}\)}
\label{alg:restart_arnoldi_like_framework}
\begin{algorithmic}[1]
\REQUIRE{Matrix \(\Mat{A}\), vector \(\Mat{b}\), function \(f\), subspace dimension \(m\)}
\ENSURE{\(\Mat{f}^{[k]} \approx f(\Mat{A})\Mat{b}\)}
\STATE{Compute an Arnoldi-like decomposition
\(\Mat{A}\Basis_{m}^{[0]} = \Basis_{m + 1}^{[0]}\uHessen_{m}^{[0]}\)
with \(\Mat{b} = \basis_{1}^{[0]}\beta^{[0]}\)}
\STATE{\(\Mat{f}^{[0]}
=
\Basis_{m}^{[0]}f(\Hessen_{m}^{[0]})\Mat{e}_{1}\beta^{[0]}\)}
\FOR{\(k = 1, 2, \dotsc\)}
\STATE{Determine the error function \(\errfunc^{[k]}(z)\)
using Corollary~\ref{cor:err_arnoldi_like_approx_k_restart}}
\STATE{Compute an Arnoldi-like decomposition
\(\Mat{A}\Basis_{m}^{[k]} = \Basis_{m + 1}^{[k]}\uHessen_{m}^{[k]}\)
with \(\basis_{m + 1}^{[k - 1]} = \basis_{1}^{[k]}\beta^{[k]}\)}
\STATE{Use adaptive quadrature to compute
\(\Mat{g}^{[k]}
\approx
\Basis_{m}^{[k]}
\errfunc^{[k]}(\Hessen_{m}^{[k]})
\Mat{e}_{1}\beta^{[k]}\)}
\STATE{Update
\(\Mat{f}^{[k]} = \Mat{f}^{[k - 1]} + \Mat{g}^{[k]}\)}
\ENDFOR
\end{algorithmic}
\end{algorithm}

\begin{corollary}[Error of the Arnoldi-like approximation after \(k\) restarts]
\label{cor:err_arnoldi_like_approx_k_restart}
Under the assumptions of Theorem~\ref{thm:err_arnoldi_like_approx}, suppose that the \(\ell\)-th cycle generates an Arnoldi-like decomposition
\(\Mat{A}\Basis_{m}^{[\ell]}
=
\Basis_{m + 1}^{[\ell]}\uHessen_{m}^{[\ell]}\)
for \(0 \leq \ell \leq k\).
Here, \(\Mat{b} = \basis_{1}^{[0]}\beta^{[0]}\) and
\(\basis_{m + 1}^{[\ell - 1]}
=
\basis_{1}^{[\ell]}\beta^{[\ell]}\)
for \(1 \leq \ell \leq k\).
Assume that
\(\spec(\Hessen_{m}^{[\ell]})
=
\{\theta_{j}^{[\ell]}\}_{j = 1}^{m}
\subset \Omega\), and define
\(\phi_{m}^{[\ell]}(t)
=
\prod_{j = 1}^{m}(t - \theta_{j}^{[\ell]})\)
and
\(\gamma_{m}^{[\ell]}
=
\prod_{j = 1}^{m}\hessen_{j + 1, j}^{[\ell]}\)
for \(0 \leq \ell \leq k\).
Then the error after \(k\) restarts is
\begin{equation}
\label{eq:err_arnoldi_like_approx_k_restart}
\begin{aligned}
f(\Mat{A})\Mat{b} - \Mat{f}^{[k]}
&=
\biggl(
\prod_{\ell = 0}^{k}\gamma_{m}^{[\ell]}\beta^{[\ell]}
\biggr)
\biggl(
\int_{\Gamma}
\frac{g(t)}
{\prod_{\ell = 0}^{k}\phi_{m}^{[\ell]}(t)}
(t\Mat{I} - \Mat{A})^{-1}
\diff t
\biggr)
\basis_{m + 1}^{[k]}
\\
&\eqcolon
\errfunc^{[k + 1]}(\Mat{A})\basis_{m + 1}^{[k]}.
\end{aligned}
\end{equation}
\end{corollary}

Following the quadrature approach of~\cite{Frommer_Guettel_Schweitzer_2014},
we approximate the reduced error vector by
\begin{equation*}
\errfunc^{[k]}(\Hessen_{m}^{[k]})\Mat{e}_{1}\beta^{[k]}
\approx
\biggl(
\prod_{\ell = 0}^{k - 1}\gamma_{m}^{[\ell]}\beta^{[\ell]}
\biggr)
\sum_{j = 1}^{r}
\frac{g(t_{j})}
{\prod_{\ell = 0}^{k - 1}\phi_{m}^{[\ell]}(t_{j})}
\omega_{j}^{[k]}
(t_{j}\Mat{I}_{m} - \Hessen_{m}^{[k]})^{-1}
\Mat{e}_{1}\beta^{[k]},
\end{equation*}
where \(t_{j}\) are quadrature nodes 
and \(\omega_{j}^{[k]}\) are weights.
In practice, the error functions are evaluated 
using adaptive quadrature.
For a given error function, 
we compute two quadrature approximations 
using \(r_{1}\) and \(r_{2}\) nodes, 
where \(r_{2} = 2r_{1}\).
We accept the result with \(r_{2}\) nodes 
if the reduced vectors differ 
by less than the prescribed tolerance times its norm.
Otherwise, we set \(r_{1} = r_{2}\), 
double \(r_{2}\), 
and repeat the test 
until the tolerance is satisfied 
or a prescribed maximum number of nodes is reached.
At each subsequent restart cycle, 
we try halving \(r_{1}\), subject to the minimum node count, 
set \(r_{2} = 2r_{1}\), 
and increase the node counts if necessary.

\subsection{Restarting Krylov and sketched Krylov subspace methods}
\label{subsec:restart_krylov_skrylov}

Under the orthogonality conditions 
specified in Section~\ref{sec:krylov_skrylov}, 
the four Krylov and sketched Krylov approximations 
share the error representation 
and quadrature-based restarting formula derived above.
This observation forms the basis of the sketch-and-restart framework developed in the next section.

\section{Sketch-and-restart algorithms for matrix functions}
\label{sec:sketching_quad_restarting}

The previous section showed that 
quadrature-based restarting applies to 
any Arnoldi-like approximation of the form 
\(\Basis_{m}f(\Hessen_{m})\Mat{e}_{1}\beta\), 
provided that the required orthogonality condition is satisfied.
In this section, we instantiate this restarting mechanism 
using different strategies for generating Arnoldi-like decompositions.
Specifically, we consider three such strategies: 
the existing sketched Arnoldi process~(Algorithm~\ref{alg:sketched_arnoldi}), 
a new sketched harmonic Arnoldi process introduced in this work~(Algorithm~\ref{alg:sketched_harmonic_arnoldi}), 
and an adaptive truncated Arnoldi process combined with a rank-\(1\) update~(Algorithm~\ref{alg:truncated_arnoldi}).
The first two are used in the fixed-dimensional restarted method~(Algorithm~\ref{alg:restart_skrylov}),
while the third is used in the adaptive restarted method~(Algorithm~\ref{alg:restart_krylov_adaptive}).
Table~\ref{tab:algorithm_overview} summarizes the restarted methods.

\begin{table}[!htbp]
\centering
\caption{Overview of the restarted methods, all returning \(\Mat{f}^{[k]} \approx f(\Mat{A})\Mat{b}\).}
\label{tab:algorithm_overview}
\small
\setlength{\tabcolsep}{3pt}
\begin{tabular}{@{}p{0.15\linewidth}
                  p{0.65\linewidth}
                  p{0.15\linewidth}@{}}
\toprule
Source & Method & Dimension \\
\midrule
\cite{Frommer_Guettel_Schweitzer_2014}
&
Restarted Arnoldi method
&
Fixed
\\
\addlinespace
\cite{FGS14b}
&
Restarted harmonic Arnoldi method 
&
Fixed
\\
\addlinespace
Algorithm~\ref{alg:restart_skrylov}
&
Restarted sketched Arnoldi or restarted sketched
harmonic Arnoldi methods
&
Fixed
\\
\addlinespace
Algorithm~\ref{alg:restart_krylov_adaptive}
&
Restarted adaptive truncated Arnoldi method with rank-\(1\) update
&
Adaptive
\\
\bottomrule
\end{tabular}
\end{table}

\subsection{Basis generation}
\label{subsec:basis_generation}

The standard Arnoldi process~\cite{Saad_2003} 
generates an orthonormal basis 
for the Krylov subspace \(\Kry_{m}(\Mat{A}, \Mat{b})\), 
but its orthogonalization and storage costs 
can become substantial when \(m\) is large.
In the sketch-and-restart framework, 
full orthogonality is not required.
It suffices to construct an Arnoldi-like decomposition 
that satisfies the orthogonality condition 
required by the corresponding approximation.
We therefore consider three basis generation strategies 
that produce possibly nonorthonormal bases 
with different computational and storage trade-offs.

The first strategy is the 
\emph{sketched Arnoldi~(s-Arnoldi) process}~\cite{Balabanov_Grigori_2022, deDamas_Grigori_Simunec_Timsit_2025}.
It is based on the Arnoldi process 
but computes the Gram--Schmidt coefficients 
in the sketched space.
The same coefficients are then used 
to update the full-space basis vectors, 
as summarized in Algorithm~\ref{alg:sketched_arnoldi}.
We use a modified Gram--Schmidt
formulation in Algorithm~\ref{alg:sketched_arnoldi}
and assume that no breakdown occurs for simplicity.
In exact arithmetic, these paired updates preserve
\(\skbasis_{j + 1} \perp \skBasis_{j}\) 
and \(\skbasis_{j + 1} = \Mat{S} \basis_{j + 1}\).
The algorithm yields a standard Arnoldi decomposition 
in the sketched space, 
and the resulting Arnoldi-like decomposition satisfies 
\(\basis_{m + 1} \perp_{\Mat{S}} \Basis_{m}\).
Therefore, the sArnoldi approximation can be computed directly from~\eqref{eq:s_arnoldi_approx_last_sorth}.
In floating-point arithmetic, however, 
sketch orthogonality may deteriorate 
and the stored
\(\skbasis_{j + 1}\) may differ 
from \(\Mat{S} \basis_{j + 1}\).
Reorthogonalization and recomputing the sketch
are possible stabilization measures~\cite{Balabanov_Grigori_2022,
deDamas_Grigori_Simunec_Timsit_2025}.
The complexity of the s-Arnoldi process 
under complex arithmetic is
\(m T_{\apply} + m T_{\sketching}
+ 4 N m^{2}
+ \bigO(s m^{2} + (N + s)m)\),
where \(T_{\apply}\) and \(T_{\sketching}\) 
denote the costs of applying \(\Mat{A}\) and \(\Mat{S}\) 
to a vector, respectively.
For fixed \(m\), 
when the application of \(\Mat{A}\) and \(\Mat{S}\) are cheap, 
and the computation in the sketched space is inexpensive,
the dominant cost is \(4 N m^{2}\).
Consequently, s-Arnoldi nearly halves the orthogonalization cost of the standard Arnoldi process, which is \(8 N m^{2}\).

\begin{algorithm}[!htbp]
\caption{Sketched Arnoldi~(s-Arnoldi) process}
\label{alg:sketched_arnoldi}
\begin{algorithmic}[1]
\FUNCTION{
\([\beta, \Basis_{m + 1}, \skBasis_{m + 1}, \uHessen_{m}] 
= 
\sArnoldi(\Mat{S}, \Mat{A}, \Mat{b}, \Mat{p}, m)\)}
\REQUIRE{
Sketching matrix \(\Mat{S}\),
matrix \(\Mat{A}\),
vector \(\Mat{b}\), 
sketched vector \(\Mat{p} = \Mat{S} \Mat{b}\),
subspace dimension \(m\)
}
\ENSURE{
Arnoldi-like decomposition
\(\Mat{A}\Basis_{m} = \Basis_{m + 1}\uHessen_{m}\)
where \(\Mat{b} = \basis_{1}\beta\) 
and 
\(\skBasis_{m + 1} = \Mat{S} \Basis_{m + 1}\)
is orthonormal}
\STATE{
\(\tbasis_{1} = \Mat{b}\), 
\(\tskbasis_{1} = \Mat{p}\)}
\STATE{
\(\beta = \|\tskbasis_{1}\|\),
\(\skbasis_{1} = \tskbasis_{1} / \beta\),
\(\basis_{1} = \tbasis_{1} / \beta\)}
\FOR{\(j = 1, 2, \dotsc, m\)}
\STATE{\(\tbasis_{j + 1} = \Mat{A}\basis_{j}\),
\(\tskbasis_{j + 1} = \Mat{S}\tbasis_{j + 1}\)}
\FOR{\(i = 1, \dotsc, j\)}
\STATE{\(\hessen_{i, j}
= \langle \tskbasis_{j + 1}, \skbasis_{i} \rangle\),
\(\tskbasis_{j + 1}
= \tskbasis_{j + 1} - \skbasis_{i}\hessen_{i, j}\),
\(\tbasis_{j + 1}
= \tbasis_{j + 1} - \basis_{i}\hessen_{i, j}\)}
\ENDFOR
\STATE{\(\hessen_{j + 1, j} = \|\tskbasis_{j + 1}\|\),
\(\skbasis_{j + 1}
= \tskbasis_{j + 1}/\hessen_{j + 1, j}\),
\(\basis_{j + 1}
= \tbasis_{j + 1}/\hessen_{j + 1, j}\)}
\ENDFOR
\end{algorithmic}
\end{algorithm}

The second strategy is the \emph{sketched harmonic Arnoldi~(sHm-Arnoldi) process}.
In~\cite[Section~6]{FGS14b}, the reduced matrix \(\Hessen_{m}\) is constructed by the standard Arnoldi process followed by a rank-\(1\) update.
In constrast, our new proposed sketched harmonic Arnoldi process constructs \(\Hessen_{m}\) directly, without requiring a rank-\(1\) update.
Its non-sketched counterpart, the \emph{harmonic Arnoldi process}, which is also first proposed in this work, constructs \(\Hessen_{m}\) directly as well.
Algorithm~\ref{alg:sketched_harmonic_arnoldi} summarizes the sHm-Arnoldi process,
where \(\{\basis_{1}, \dotsc, \basis_{j}\}\) 
and \(\{\Abasis_{1}, \dotsc, \Abasis_{j}\}\)
denote the bases of \(\Kry_{j}(\Mat{A}, \Mat{b})\) and \(\Kry_{j}(\Mat{A}, \Mat{A}\Mat{b})\), respectively,
and \(\{\skbasis_{1}, \dotsc, \skbasis_{j}\}\) 
and \(\{\skAbasis_{1}, \dotsc, \skAbasis_{j}\}\) 
denote their sketches. 
As in Algorithm~\ref{alg:sketched_arnoldi},
Algorithm~\ref{alg:sketched_harmonic_arnoldi}
uses sequential paired updates in the sketched and full spaces,
and we assume that no breakdown occurs for simplicity.
The scaled images
\(\Abasis_{i} = \Mat{A}\basis_{i}/\rho_{i}^{\herm}\)
and their sketches
\(\skAbasis_{i} = \Mat{S}\Abasis_{i}\)
satisfy
\(\langle \skbasis_{i}, \skAbasis_{i} \rangle = 1\)
in exact arithmetic.
At each inner iteration,
the coefficient is computed 
from the current sketched vector \(\tskbasis_{j + 1}\) 
using \(\skAbasis_{i}\) as the test vector,
and the same coefficient is used to update
the sketched and full-space vectors \(\skbasis_{j + 1}\) and \(\basis_{j + 1}\).
In exact arithmetic, these paired updates preserve
\(\skbasis_{j + 1} = \Mat{S}\basis_{j + 1}\)
and enforce
\(\skbasis_{j + 1} \perp \skABasis_{j}\).

In particular, Algorithm~\ref{alg:sketched_harmonic_arnoldi} returns a nonorthonormal basis 
\(\Basis_{m + 1}\) satisfying 
\(\basis_{j + 1} \perp_{\Mat{S}} \Mat{A}\Basis_{j}\) for \(1 \leq j \leq m\).
The resulting decomposition allows the sHmArnoldi approximation 
to be computed directly from~\eqref{eq:s_hm_arnoldi_approx_last_shm_orth}.
Setting \(\Mat{S} = \Mat{I}_{N}\) and remove the sketched variables in Algorithm~\ref{alg:sketched_harmonic_arnoldi}
gives the harmonic Arnoldi~(Hm-Arnoldi) process.
To the best of our knowledge, 
both the harmonic Arnoldi and sketched harmonic Arnoldi processes are new.
The term ``harmonic'' refers to 
harmonic Ritz projection 
with test space \(\Mat{A}\Kry_{m}(\Mat{A}, \Mat{b})\) 
and target \(0\)~\cite[Section~6]{FGS14b};
the eigenvalues of \(\Hessen_{m}\) are the corresponding sketched harmonic Ritz values.

The complexity of the sHm-Arnoldi process
under complex arithmetic is
\(m T_{\apply} + m T_{\sketching}
+ 4 N m^{2}
+ \bigO(s m^{2} + (N + s)m)\).
Under the same cost assumptions as for the s-Arnoldi process,
the dominant cost is \(4 N m^{2}\).
Thus, sHm-Arnoldi also nearly halves
the orthogonalization cost of the standard Arnoldi process.

\begin{algorithm}[!htbp]
\caption{Sketched harmonic Arnoldi~(sHm-Arnoldi) process}
\label{alg:sketched_harmonic_arnoldi}
\begin{algorithmic}[1]
\FUNCTION{
\([\beta, \Basis_{m + 1}, \skBasis_{m + 1},\uHessen_{m}] 
= 
\sHmArnoldi(\Mat{S}, \Mat{A}, \Mat{b}, \Mat{p}, m)\)}
\REQUIRE{
Sketching matrix \(\Mat{S}\),
matrix \(\Mat{A}\),
vector \(\Mat{b}\), 
sketched vector \(\Mat{p} = \Mat{S} \Mat{b}\),
subspace dimension \(m\)
}
\ENSURE{
Arnoldi-like decomposition
\(\Mat{A}\Basis_{m} = \Basis_{m + 1}\uHessen_{m}\)
where \(\Mat{b} = \basis_{1}\beta\),
\(\skBasis_{m + 1} = \Mat{S} \Basis_{m + 1}\),
and
\(\basis_{j + 1} \perp_{\Mat{S}} \Mat{A} \Basis_{j}\) for \(1 \leq j \leq m\)}
\STATE{
\(\tbasis_{1} = \Mat{b}\),
\(\tskbasis_{1} = \Mat{p}\)}
\STATE{
\(\beta = \|\tskbasis_{1}\|\),
\(\skbasis_{1} = \tskbasis_{1}/\beta\),
\(\basis_{1} = \tbasis_{1}/\beta\)}
\FOR{\(j = 1, 2, \dotsc, m\)}
\STATE{
\(\tbasis_{j + 1} = \Mat{A}\basis_{j}\),
\(\tskbasis_{j + 1} = \Mat{S}\tbasis_{j + 1}\)}
\STATE{
\(\rho_{j}
= \langle \skbasis_{j}, \tskbasis_{j + 1} \rangle\)}
\STATE{
\(\Abasis_{j} = \tbasis_{j + 1}/\rho_{j}^{\herm}\),
\(\skAbasis_{j} = \tskbasis_{j + 1}/\rho_{j}^{\herm}\)}
\FOR{\(i = 1, \dotsc, j\)}
\STATE{
\(\hessen_{i, j} = \langle \tskbasis_{j + 1}, \skAbasis_{i} \rangle\),
\(\tskbasis_{j + 1}
= \tskbasis_{j + 1} - \skbasis_{i}\hessen_{i, j}\),
\(\tbasis_{j + 1}
= \tbasis_{j + 1} - \basis_{i}\hessen_{i, j}\)}
\ENDFOR
\STATE{
\(\hessen_{j + 1, j} = \|\tskbasis_{j + 1}\|\),
\(\skbasis_{j + 1}
= \tskbasis_{j + 1}/\hessen_{j + 1, j}\),
\(\basis_{j + 1}
= \tbasis_{j + 1}/\hessen_{j + 1, j}\)
}
\ENDFOR
\end{algorithmic}
\end{algorithm}

\begin{theorem}[Properties of the sHm-Arnoldi process]
\label{thm:shm_arnoldi_basis}
Suppose that Algorithm~\ref{alg:sketched_harmonic_arnoldi}
is performed in exact arithmetic, with
\(\beta \neq 0\),
\(\rho_{j} \neq 0\), and
\(\hessen_{j + 1, j} \neq 0\)
for \(1 \leq j \leq m\).
Then,
\(\Basis_{j}\) is a basis of \(\Kry_{j}(\Mat{A}, \Mat{b})\) 
for \(1 \leq j \leq m + 1\)
and
\(\ABasis_{j}\) is a basis of \(\Mat{A}\Kry_{j}(\Mat{A}, \Mat{b})\) for \(1 \leq j \leq m\).
The generated factors satisfy
\begin{equation*}
\begin{gathered}
\Mat{A}\Basis_{m}
=
\Basis_{m + 1}\uHessen_{m},
\qquad
\Mat{b} = \basis_{1}\beta,
\\
\skBasis_{m + 1} = \Mat{S}\Basis_{m + 1},
\qquad
\skABasis_{m} = \Mat{S}\ABasis_{m}
\end{gathered}
\end{equation*}
and
\begin{equation*}
\basis_{j + 1}
\perp_{\Mat{S}}
\Mat{A}\Basis_{j},
\qquad
1 \leq j \leq m.
\end{equation*}
\end{theorem}

\begin{proof}
We prove the result from a projection perspective.
The paired updates and scaling give
\(\skBasis_{m + 1} = \Mat{S}\Basis_{m + 1}\),
\(\skABasis_{m} = \Mat{S}\ABasis_{m}\),
\(\|\skbasis_{i}\| = 1\) 
for \(1 \leq i \leq m + 1\),
and
\(\langle \skbasis_{i}, \skAbasis_{i} \rangle = 1\) 
for \(1 \leq i \leq m\).
Since
\(\Abasis_{i} = \Mat{A}\basis_{i}/\rho_{i}^{\herm}\),
the required orthogonality is equivalent to
\begin{equation*}
\basis_{j + 1} \perp_{\Mat{S}} \Mat{A}\Basis_{j}
\quad\Longleftrightarrow\quad
\skbasis_{j + 1} \perp \skABasis_{j}.
\end{equation*}

We proceed by induction, starting from
\(\basis_{1} = \Mat{b}/\beta\).
Suppose that \(\Basis_{j}\) is a basis of
\(\Kry_{j}(\Mat{A}, \Mat{b})\) and
\(\langle \skbasis_{i}, \skAbasis_{k} \rangle = 0\)
for \(k < i \leq j\).
Then
\begin{equation*}
\Mat{T}_{j}
=
\skABasis_{j}^{\herm} \skBasis_{j}
=
\bigl(
\langle \skbasis_{k}, \skAbasis_{i} \rangle
\bigr)_{i,k = 1}^{j} \in \complex^{j \times j}
\end{equation*}
is unit lower triangular with diagonal entries \(\langle \skbasis_{i}, \skAbasis_{i} \rangle = 1\).
Its invertibility implies that both
\(\skABasis_{j}\) and \(\ABasis_{j}\)
have full column rank.
Since each \(\Abasis_{i}\) is a 
nonzero multiple of \(\Mat{A}\basis_{i}\), 
\(\ABasis_{j}\) is therefore a basis of \(\Mat{A}\Kry_{j}(\Mat{A}, \Mat{b})\).

Define
\begin{equation*}
    \Mat{\Pi}_{j}
    =
    \Mat{I}_{N}
    -
    \Basis_{j}\Mat{T}_{j}^{-1}
    \skABasis_{j}^{\herm}\Mat{S}.
\end{equation*}
This is an oblique projector satisfying
\begin{equation*}
\Mat{\Pi}_{j}^{2} = \Mat{\Pi}_{j},
\qquad
\Mat{\Pi}_{j}\Basis_{j} = \Mat{0},
\qquad
\Mat{S}\Mat{\Pi}_{j}
=
\bigl(
\Mat{I}_{s}
-
\skBasis_{j}\Mat{T}_{j}^{-1}
\skABasis_{j}^{\herm}
\bigr)\Mat{S}.
\end{equation*}
In the inner loop, applying the successive sketched updates
to any \(\Mat{w} \in \complex^{s}\) produces coefficients
satisfying
\begin{equation*}
c_{i}
=
\left\langle
\Mat{w}
-
\sum_{k = 1}^{i - 1} \skbasis_{k} c_{k},
\skAbasis_{i}
\right\rangle,
\qquad 1 \leq i \leq j.
\end{equation*}
Since \(\Mat{T}_{j}\) is unit lower triangular,
these are precisely the forward substitution relations for
\(\Mat{T}_{j} \Mat{c}
= \skABasis_{j}^{\herm}\Mat{w}\),
where \(\Mat{c} = [c_{1}, \dotsc, c_{j}]^{\trans}\).
Thus, the resulting vector after \(j\) updates is
\begin{equation*}
\Mat{w} - \skBasis_{j} \Mat{c}
=
\bigl(
    \Mat{I}_{s}
    -
    \skBasis_{j}\Mat{T}_{j}^{-1}
    \skABasis_{j}^{\herm}
\bigr)\Mat{w}.
\end{equation*}
As \(\Mat{w}\) is arbitrary, this proves
\begin{equation*}
(\Mat{I}_{s}
- \skbasis_{j}\skAbasis_{j}^{\herm})
\dotsm
(\Mat{I}_{s}
- \skbasis_{1}\skAbasis_{1}^{\herm})
=
\Mat{I}_{s}
-
\skBasis_{j}\Mat{T}_{j}^{-1}
\skABasis_{j}^{\herm}.
\end{equation*}
The paired updates and normalization therefore give
\begin{equation*}
\basis_{j + 1}
=
\frac{\Mat{\Pi}_{j}\Mat{A}\basis_{j}}{\hessen_{j + 1,j}},
\qquad
\skbasis_{j + 1}
=
\frac{\Mat{S}\Mat{\Pi}_{j}\Mat{A}\basis_{j}}{\hessen_{j + 1,j}}
\perp \skABasis_{j},
\end{equation*}
where the last orthogonality follows from
\begin{equation*}
\skABasis_{j}^{\herm} \Mat{S} \Mat{\Pi}_{j}
=
\skABasis_{j}^{\herm} \bigl(
\Mat{I}_{s}
-
\skBasis_{j}\Mat{T}_{j}^{-1}
\skABasis_{j}^{\herm}
\bigr)\Mat{S}
=
\Mat{0}.
\end{equation*}
Hence,
\(\basis_{j + 1}
\perp_{\Mat{S}} \Mat{A}\Basis_{j}\).

Since \(\|\skbasis_{j + 1}\| = 1\),
\(\basis_{j + 1}\) is nonzero, and the idempotence
of \(\Mat{\Pi}_{j}\) gives
\begin{equation*}
    \Mat{\Pi}_{j}\basis_{j + 1}
    =
    \basis_{j + 1}
    \neq
    \Mat{0}.
\end{equation*}
Since \(\Mat{\Pi}_{j}\Basis_{j} = \Mat{0}\),
the new vector cannot belong to
\(\Kry_{j}(\Mat{A}, \Mat{b})\).
Moreover, the update formula gives
\(\basis_{j + 1} \in \Kry_{j + 1}(\Mat{A}, \Mat{b})\).
Thus, \(\Basis_{j + 1}\) is a basis of
\(\Kry_{j + 1}(\Mat{A}, \Mat{b})\),
completing the induction.
Finally, collecting the recurrence coefficients gives
\(\Mat{A}\Basis_{m}
= \Basis_{m + 1}\uHessen_{m}\).
\end{proof}

The third strategy is the
\emph{truncated Arnoldi~(\(t\)-Arnoldi) process}~\cite{Saad_2003}.
It generates a nonorthonormal basis
by orthogonalizing each new vector
only against the \(t\) most recently generated basis vectors.
Let \(m_{\mathrm{out}} \leq m_{\max}\)
denote the adaptively determined output dimension.
For \(j < m_{\max}\),
whenever the sketching dimension \(s\) satisfies
\(s < \eta(j + 1)\),
we increase \(s\) by \(s_{0}\),
draw a new sketching matrix
\(\Mat{S} \in \complex^{s \times N}\),
and recompute
\(\skBasis_{j + 1} = \Mat{S}\Basis_{j + 1}\).
After any such resketching,
we check \(\cond(\skBasis_{j + 1})\)
at steps \(j = 5, 10, \dotsc < m_{\max}\)
and stop if it exceeds \(\tau_{\condcontrol}\).
The process also terminates when the dimension reaches \(m_{\max}\).

At the end of basis generation,
we apply a rank-\(1\) update to the Arnoldi-like decomposition
to enforce a prescribed orthogonality condition
on the last basis vector;
see Section~\ref{subsec:rank1_update}.
When the condition is
\(\basis_{m_{\mathrm{out}} + 1}
\perp \Basis_{m_{\mathrm{out}}}\),
we call the resulting method
the adaptive Arnoldi~(adaArnoldi-\(t\)) method.
When the condition is
\(\basis_{m_{\mathrm{out}} + 1}
\perp \Mat{A}\Basis_{m_{\mathrm{out}}}\),
we call it
the adaptive harmonic Arnoldi~(adaHmArnoldi-\(t\)) method.

The complete procedure is summarized in Algorithm~\ref{alg:truncated_arnoldi}.
As in Algorithms~\ref{alg:sketched_arnoldi}
and~\ref{alg:sketched_harmonic_arnoldi},
we assume that no breakdown occurs for simplicity,
including in the final rank-\(1\) update.
For a small fixed truncation parameter \(t\),
the complexity under complex arithmetic, excluding resketching, is
\begin{equation*}
m_{\mathrm{out}} T_{\apply}
+ m_{\mathrm{out}} T_{\sketching}
+ 8 N m_{\mathrm{out}}^{2}
+ \bigO((N + s)m_{\mathrm{out}}
+ m_{\mathrm{out}}^{3}).
\end{equation*}
Under the same cost assumptions as for the preceding processes,
the dominant cost is \(8 N m_{\mathrm{out}}^{2}\).
Since the resulting Krylov subspace dimension 
is typically small,
the rank-\(1\) update is inexpensive 
and can be performed at the BLAS-2 level.

\begin{algorithm}[!htbp]
\caption{Truncated Arnoldi~(\(t\)-Arnoldi) process}
\label{alg:truncated_arnoldi}
\begin{algorithmic}[1]
\FUNCTION{
\([\Mat{S}, m_{\mathrm{out}}, \beta,
\Basis_{m_{\mathrm{out}} + 1},
\skBasis_{m_{\mathrm{out}} + 1},
\uHessen_{m_{\mathrm{out}}}]\)
}
\CONTSTATE{
\(=
\tArnoldi
(\Mat{S},\Mat{A},\Mat{b},\Mat{p},
m_{\max},t,\tau_{\condcontrol},
s_{0},\eta,\orthcondition)\)}
\REQUIRE{
Sketching matrix \(\Mat{S}\),
matrix \(\Mat{A}\),
vector \(\Mat{b}\),
sketched vector \(\Mat{p} = \Mat{S}\Mat{b}\),
maximum subspace dimension \(m_{\max}\),
truncation parameter \(t\),
condition number threshold \(\tau_{\condcontrol}\),
sketching dimension increment \(s_{0}\),
sketching dimension control parameter \(\eta\),
and selected orthogonality condition \(\orthcondition\)
for the rank-\(1\) update from
Section~\ref{subsec:rank1_update}}
\ENSURE{
Adaptively determined dimension
\(m_{\mathrm{out}} \leq m_{\max}\),
updated sketching matrix~\(\Mat{S}\),
and Arnoldi-like decomposition
\(\Mat{A}\Basis_{m_{\mathrm{out}}}
=
\Basis_{m_{\mathrm{out}} + 1}
\uHessen_{m_{\mathrm{out}}}\),
where \(\Mat{b} = \basis_{1}\beta\),
\(\skBasis_{m_{\mathrm{out}} + 1}
=
\Mat{S}\Basis_{m_{\mathrm{out}} + 1}\),
and \(\basis_{m_{\mathrm{out}} + 1}\)
satisfies \(\orthcondition\)}
\STATE{
\(\beta = \|\Mat{b}\|\),
\(\basis_{1} = \Mat{b}/\beta\),
\(\skbasis_{1} = \Mat{p}/\beta\)}
\STATE{
Let \(s\) be the number of rows of \(\Mat{S}\),
and initialize
\(\skBasis_{1} = [\skbasis_{1}]\)}
\FOR{\(j = 1, 2, \dotsc, m_{\max}\)}
\STATE{
\(\tbasis_{j + 1} = \Mat{A}\basis_{j}\)}
\FOR{\(i = \max(1, j - t + 1), \dotsc, j\)}
\STATE{
\(\hessen_{i, j}
=
\langle \tbasis_{j + 1}, \basis_{i} \rangle\),
\(\tbasis_{j + 1}
=
\tbasis_{j + 1} - \basis_{i}\hessen_{i, j}\)}
\ENDFOR
\STATE{
\(\hessen_{j + 1, j} = \|\tbasis_{j + 1}\|\),
\(\basis_{j + 1}
=
\tbasis_{j + 1}/\hessen_{j + 1, j}\)}
\STATE{
\(\skbasis_{j + 1} = \Mat{S}\basis_{j + 1}\),
\(\skBasis_{j + 1}
=
[\skBasis_{j}, \skbasis_{j + 1}]\)}
\IF{\(j < m_{\max}\)}
\IF{\(s < \eta(j + 1)\)}
\STATE{\(s \gets s + s_{0}\)}
\STATE{
Draw a new sketching matrix
\(\Mat{S} \in \complex^{s \times N}\)}
\STATE{
\(\skBasis_{j + 1}
\gets
\Mat{S}\Basis_{j + 1}\)}
\ENDIF
\IF{\(\cond(\skBasis_{j + 1}) > \tau_{\condcontrol}\)}
\STATE{Break}
\ENDIF
\ENDIF
\ENDFOR
\STATE{Set \(m_{\mathrm{out}} = j\)}
\STATE{
Compute the rank-\(1\) update from
Section~\ref{subsec:rank1_update} for \(\orthcondition\),
using
\(\alpha_{m_{\mathrm{out}} + 1}
=
\|\basis_{m_{\mathrm{out}} + 1}
-
\Basis_{m_{\mathrm{out}}}\Mat{c}_{m_{\mathrm{out}}}\|\)}
\STATE{
\(\skbasis_{m_{\mathrm{out}} + 1}
\gets
\bigl(
\skbasis_{m_{\mathrm{out}} + 1}
-
\skBasis_{m_{\mathrm{out}}}\Mat{c}_{m_{\mathrm{out}}}
\bigr)/\alpha_{m_{\mathrm{out}} + 1}\)}
\STATE{
Replace
\(\Hessen_{m_{\mathrm{out}}}\),
\(\basis_{m_{\mathrm{out}} + 1}\), and
\(\hessen_{m_{\mathrm{out}} + 1, m_{\mathrm{out}}}\)
by their updated values}
\end{algorithmic}
\end{algorithm}

\subsection{Algorithms}
\label{subsec:algorithms}

We now present two algorithmic realizations of the sketch-and-restart framework.
Algorithm~\ref{alg:restart_skrylov} uses a fixed Krylov subspace dimension and applies either the s-Arnoldi or the sHm-Arnoldi process in each restart cycle, leading respectively to the restarted sArnoldi and restarted sHmArnoldi methods.
The error function is updated as in Algorithm~\ref{alg:restart_arnoldi_like_framework}, and the correction term is evaluated by adaptive quadrature.

Algorithm~\ref{alg:restart_krylov_adaptive} instead chooses the Krylov subspace dimension adaptively using the \(t\)-Arnoldi process.
The dimension \(m_{0}\) obtained in the first restart cycle is used as the upper bound \(m_{\max}\) in all subsequent cycles,
so that the sketching matrix do not need to be redrawn
ins subsequent cycles.
The same quadrature-based restarting mechanism is then applied to the resulting Arnoldi-like decompositions.

\begin{algorithm}[!htbp]
\caption{Restarted sketched Krylov subspace methods
for computing \(f(\Mat{A})\Mat{b}\)}
\label{alg:restart_skrylov}
\begin{algorithmic}[1]
\REQUIRE{
Matrix \(\Mat{A}\), 
vector \(\Mat{b}\), 
function \(f\),
subspace dimension \(m\), 
sketching dimension \(s\),
choice of the s-Arnoldi or sHm-Arnoldi process,
stopping tolerance \(\tau\),
quadrature tolerance \(\tau_{\quadature}\),
and maximum number of quadrature nodes \(n_{\quadature, \max}\)}
\ENSURE{\(\Mat{f}^{[k]} \approx f(\Mat{A})\Mat{b}\)}
\STATE{
Draw a sketching matrix
\(\Mat{S} \in \complex^{s \times N}\)
and set \(\Mat{p} = \Mat{S}\Mat{b}\)}
\STATE{
Compute
\(\beta^{[0]}\),
\(\Basis_{m + 1}^{[0]}\),
\(\skBasis_{m + 1}^{[0]}\), and
\(\uHessen_{m}^{[0]}\)
by calling \(\sArnoldi\) or \(\sHmArnoldi\)
with inputs
\((\Mat{S},\Mat{A},\Mat{b},\Mat{p},m)\),
according to the selected process}
\STATE{
\(\Mat{f}^{[0]}
=
\Basis_{m}^{[0]}
f(\Hessen_{m}^{[0]})
\Mat{e}_{1}\beta^{[0]}\)}
\FOR{\(k = 1,2,\dotsc\)}
\STATE{
Determine the error function \(\errfunc^{[k]}(z)\)
using Corollary~\ref{cor:err_arnoldi_like_approx_k_restart}}
\STATE{
Compute
\(\beta^{[k]}\),
\(\Basis_{m + 1}^{[k]}\),
\(\skBasis_{m + 1}^{[k]}\), and
\(\uHessen_{m}^{[k]}\)
by calling \(\sArnoldi\) or \(\sHmArnoldi\)
with inputs
\((\Mat{S},\Mat{A},
\basis_{m + 1}^{[k - 1]},
\skbasis_{m + 1}^{[k - 1]},m)\),
according to the selected process}
\STATE{
Compute
\(\Mat{u}^{[k]}
\approx
\errfunc^{[k]}(\Hessen_{m}^{[k]})
\Mat{e}_{1}\beta^{[k]}\)
using adaptive quadrature with tolerance
\(\tau_{\quadature}\)
and maximum number of quadrature nodes
\(n_{\quadature, \max}\)
}
\STATE{
\(\Mat{g}^{[k]}
=
\Basis_{m}^{[k]}\Mat{u}^{[k]}\)}
\STATE{
\(\Mat{f}^{[k]}
=
\Mat{f}^{[k - 1]}+\Mat{g}^{[k]}\)}
\IF{
\(\|\Mat{g}^{[k]}\|
<
\tau\,\|\Mat{f}^{[k]}\|\)}
\STATE{Break}
\ENDIF
\ENDFOR
\end{algorithmic}
\end{algorithm}

\begin{algorithm}[!htbp]
\caption{Restarted adaptive Krylov subspace methods
for computing \(f(\Mat{A})\Mat{b}\)}
\label{alg:restart_krylov_adaptive}
\begin{algorithmic}[1]
\REQUIRE{
Matrix \(\Mat{A}\), 
vector \(\Mat{b}\), 
function \(f\),
truncation parameter \(t\),
condition number threshold \(\tau_{\condcontrol}\),
maximum subspace dimension \(m_{\max}\),
sketching dimension increment \(s_{0}\),
sketching dimension control parameter \(\eta\),
orthogonality condition \(\orthcondition\)
from Section~\ref{subsec:rank1_update},
stopping tolerance \(\tau\),
quadrature tolerance \(\tau_{\quadature}\),
and maximum number of quadrature nodes
\(n_{\quadature,\max}\)}
\ENSURE{\(\Mat{f}^{[k]} \approx f(\Mat{A})\Mat{b}\)}
\STATE{
Draw a sketching matrix
\(\Mat{S} \in \complex^{s_{0} \times N}\)
and set \(\Mat{p} = \Mat{S}\Mat{b}\)}
\STATE{
\([\Mat{S},m_{0},\beta^{[0]},
\Basis_{m_{0} + 1}^{[0]},
\skBasis_{m_{0} + 1}^{[0]},
\uHessen_{m_{0}}^{[0]}]\)}
\CONTSTATE{
\(=
\tArnoldi
(\Mat{S},\Mat{A},\Mat{b},\Mat{p},
m_{\max},t,\tau_{\condcontrol},
s_{0},\eta,\orthcondition)\)}
\STATE{Set \(m_{\max} = m_{0}\)}
\STATE{
\(\Mat{f}^{[0]}
=
\Basis_{m_{0}}^{[0]}
f(\Hessen_{m_{0}}^{[0]})
\Mat{e}_{1}\beta^{[0]}\)}
\FOR{\(k = 1,2,\dotsc\)}
\STATE{
Determine the error function \(\errfunc^{[k]}(z)\)
using Corollary~\ref{cor:err_arnoldi_like_approx_k_restart}}
\STATE{
\([\Mat{S},m_{k},\beta^{[k]},
\Basis_{m_{k} + 1}^{[k]},
\skBasis_{m_{k} + 1}^{[k]},
\uHessen_{m_{k}}^{[k]}]\)}
\CONTSTATE{
\(=
\tArnoldi
(\Mat{S},\Mat{A},
\basis_{m_{k - 1} + 1}^{[k - 1]},
\skbasis_{m_{k - 1} + 1}^{[k - 1]},
m_{\max},t,\tau_{\condcontrol},
s_{0},\eta,\orthcondition)\)}
\STATE{
Compute
\(\Mat{u}^{[k]}
\approx
\errfunc^{[k]}(\Hessen_{m_{k}}^{[k]})
\Mat{e}_{1}\beta^{[k]}\)
using adaptive quadrature with tolerance
\(\tau_{\quadature}\)
and maximum number of quadrature nodes
\(n_{\quadature,\max}\)}
\STATE{
\(\Mat{g}^{[k]}
=
\Basis_{m_{k}}^{[k]}\Mat{u}^{[k]}\)}
\STATE{
\(\Mat{f}^{[k]}
=
\Mat{f}^{[k - 1]}+\Mat{g}^{[k]}\)}
\IF{
\(\|\Mat{g}^{[k]}\|
<
\tau\,\|\Mat{f}^{[k]}\|\)}
\STATE{Break}
\ENDIF
\ENDFOR
\end{algorithmic}
\end{algorithm}

\section{Convergence analysis}
\label{sec:convergence_analysis}

In this section, we analyze the convergence 
of the proposed sketch-and-restart algorithms.
We assume exact arithmetic and 
neglect quadrature errors, 
so the analysis focuses on the Krylov 
and sketching components of the methods.
The non-sketched adaptive variants 
can be related directly to 
existing restarted Arnoldi and harmonic Arnoldi methods 
through the Arnoldi-like framework developed above.
The main new issue is the convergence of 
the restarted sketched harmonic Arnoldi method, 
for which the Petrov--Galerkin condition is imposed 
only after applying the sketching matrix.

Our analysis extends the GMRES-based approach of~\cite{FGS14b} 
to the sketched setting.
This approach is particularly useful 
because it does not rely on CG-type arguments 
and therefore remains applicable in the presence of 
complex Ritz values.
We focus on Stieltjes functions of the form
\begin{equation}
\label{eq:stieltjes_func}
f(z)
=
\int_{-\infty}^{0}
\frac{g(t)}{t - z}
\diff t,
\end{equation}
where \(g\) is a suitable nonpositive function on \((-\infty, 0]\).

\subsection{Convergence of the adaptive restarted Krylov methods}
\label{subsec:hm_arnoldi}

We first consider the non-sketched adaptive variants 
of Algorithm~\ref{alg:restart_krylov_adaptive}.
The \(t\)-Arnoldi basis may be nonorthonormal; 
the final rank-\(1\) update enforces the selected Arnoldi 
or harmonic Arnoldi orthogonality condition, 
making these variants mathematically equivalent 
to the corresponding restarted methods in~\cite{FGS14b}.

Assume that we have an 
Arnoldi-like decomposition~\eqref{eq:arnoldi_like_decomp}, 
where \(\Basis_{m}\) has full column rank 
but is not necessarily orthonormal.
We further assume that \(\basis_{m + 1} \perp \Mat{A}\Basis_{m}\), 
which can be enforced by the rank-\(1\) update 
discussed in Section~\ref{subsec:rank1_update}.
Under this assumption, 
Algorithm~\ref{alg:restart_krylov_adaptive} is mathematically equivalent to 
the restarted HmArnoldi method analyzed in~\cite{FGS14b}.
The corresponding HmArnoldi approximation to the solution 
of the shifted linear system~\eqref{eq:shifted_linear_system} is
\(\Mat{x}_{m}^{\HmArn}(t)
=
\Basis_{m}(t\Mat{I}_{m} - \Hessen_{m})^{-1}\Mat{e}_{1}\beta\).
Its residual satisfies
\begin{equation*}
\begin{aligned}
\Mat{r}_{m}^{\HmArn}(t)
&=
\Mat{b}
-
(t\Mat{I} - \Mat{A})\Mat{x}_{m}^{\HmArn}(t)
\\
&=
\Mat{b}
-
(t\Mat{I} - \Mat{A})
\Basis_{m}
(t\Mat{I}_{m} - \Hessen_{m})^{-1}
\Mat{e}_{1}\beta
\\
&=
\Mat{b}
-
[\Basis_{m} \ \basis_{m + 1}]
(t\underline{\Mat{I}}_{m} - \uHessen_{m})
(t\Mat{I}_{m} - \Hessen_{m})^{-1}
\Mat{e}_{1}\beta
\\
&=
\basis_{m + 1}
\hessen_{m + 1,m}
\Mat{e}_{m}^{\trans}
(t\Mat{I}_{m} - \Hessen_{m})^{-1}
\Mat{e}_{1}\beta
\\
&\eqcolon
\basis_{m + 1}\zeta^{\HmArn}(t).
\end{aligned}
\end{equation*}
Thus, the residual is collinear with \(\basis_{m + 1}\) 
for all \(t\).

For each \(t\), let \(p_{m - 1,t}^{\HmArn} \in \poly_{m - 1}\) 
be the polynomial that interpolates \(h_{t}(z) = (t - z)^{-1}\) 
at the harmonic Ritz values \(\spec(\Hessen_{m})\).
Then Lemma~\ref{lem:poly_A_b1_arnoldi_like} gives
\begin{equation*}
\Mat{x}_{m}^{\HmArn}(t)
=
\Basis_{m}(t\Mat{I}_{m} - \Hessen_{m})^{-1}\Mat{e}_{1}\beta
=
p_{m - 1,t}^{\HmArn}(\Mat{A})\Mat{b}.
\end{equation*}
If \(f\) is a Stieltjes function 
of the form~\eqref{eq:stieltjes_func}, 
then the HmArnoldi approximation~\eqref{eq:hm_arnoldi_approx_last_hm_orth} 
can be written as
\begin{equation*}
\Mat{f}_{m}^{\HmArn}
=
\int_{-\infty}^{0}
g(t)\Mat{x}_{m}^{\HmArn}(t) \diff t
\end{equation*}.
Up to the sign convention in the Stieltjes representation, 
this is precisely the class of HmArnoldi approximants 
analyzed in~\cite{FGS14b}.
Consequently, the convergence results of~\cite{FGS14b} 
carry over to Algorithm~\ref{alg:restart_krylov_adaptive}.
The following theorem summarizes this observation.

\begin{theorem}[Convergence of the adaptive restarted Arnoldi-like method]
\label{thm:conv_adaptive_restarted_arnoldi_like}
Let \(f\) be a Stieltjes function, 
and consider Algorithm~\ref{alg:restart_krylov_adaptive}, 
where the basis is generated by 
the truncated Arnoldi process 
in Algorithm~\ref{alg:truncated_arnoldi}.

If the rank-\(1\) update enforces 
\(\basis_{m + 1} \perp \Basis_{m}\), 
then the algorithm is mathematically equivalent to 
the restarted Arnoldi method 
for computing \(f(\Mat{A})\Mat{b}\).
Hence, if \(\Mat{A}\) is Hermitian positive definite, 
the method converges for every restart length \(m \geq 1\).

If the rank-\(1\) update enforces 
\(\basis_{m + 1} \perp \Mat{A}\Basis_{m}\), 
then the algorithm is mathematically equivalent to 
the restarted HmArnoldi method 
for computing \(f(\Mat{A})\Mat{b}\).
Hence, if \(\Mat{A}\) is positive real, 
the method converges for every restart length \(m \geq 1\).
\end{theorem}

\begin{proof}
The equivalence to the restarted Arnoldi 
and restarted HmArnoldi methods 
follows from the Arnoldi-like framework 
and the rank-\(1\) update described above.
The convergence of the restarted Arnoldi method 
for Hermitian positive definite matrices 
is established in~\cite[Theorem~4.3]{FGS14b}, 
while the convergence of the restarted HmArnoldi method 
for positive real matrices 
is proved in~\cite[Theorem~6.5]{FGS14b}.
Since both convergence results hold for 
every restart length \(m \geq 1\), 
they remain valid when the restart length varies 
between restart cycles.
\end{proof}

\subsection{Sketched harmonic Arnoldi with a nonorthogonal basis}
\label{subsec:shm_arnoldi}

We now turn to the sketched setting, 
where the Petrov--Galerkin condition 
is imposed only after 
applying the sketching matrix.
Unlike the non-sketched variants discussed above, 
the restarted sHmArnoldi method 
cannot be reduced directly to the HmArnoldi method 
analyzed in~\cite{FGS14b}.
We therefore establish the residual relations 
needed to extend the GMRES-based convergence analysis 
to the sketched setting.

Suppose that, in the Arnoldi-like decomposition~\eqref{eq:arnoldi_like_decomp}, 
the matrix \(\Basis_{m}\) has full column rank 
but is not necessarily orthonormal.
Assume further that 
\(\basis_{m + 1} \perp_{\Mat{S}} \Mat{A}\Basis_{m}\).
The sHmArnoldi approximation 
to the solution of 
the shifted linear system~\eqref{eq:shifted_linear_system} is
\(\Mat{x}_{m}^{\sHmArn}(t)
=
\Basis_{m}(t\Mat{I}_{m} - \Hessen_{m})^{-1}\Mat{e}_{1}\beta\).
As in Section~\ref{subsec:hm_arnoldi}, 
its residual satisfies
\begin{equation}
\label{eq:s_hm_arnoldi_res}
\Mat{r}_{m}^{\sHmArn}(t)
=
\Mat{b}
-
(t\Mat{I} - \Mat{A})\Mat{x}_{m}^{\sHmArn}(t)
=
\basis_{m + 1}\zeta^{\sHmArn}(t).
\end{equation}
Thus, the residual is collinear with \(\basis_{m + 1}\) for all \(t\).

Similarly, for each \(t\), 
there exists a polynomial
\(p_{m - 1,t}^{\sHmArn} \in \poly_{m - 1}\) 
such that
\begin{equation*}
\Mat{x}_{m}^{\sHmArn}(t)
=
p_{m - 1,t}^{\sHmArn}(\Mat{A})\Mat{b},
\end{equation*}
where \(p_{m - 1,t}^{\sHmArn}\) interpolates
\(h_{t}(z) = (t - z)^{-1}\) 
at the sketched harmonic Ritz values \(\spec(\Hessen_{m})\).
Therefore, the residual can be written as
\begin{equation*}
\Mat{r}_{m}^{\sHmArn}(t)
=
q_{m,t}^{\sHmArn}(\Mat{A})\Mat{b},
\qquad
q_{m,t}^{\sHmArn}(z)
=
1 - (t - z)p_{m - 1,t}^{\sHmArn}(z).
\end{equation*}
The following lemma is the sketched analogue of~\cite[Lemma~6.2(iv)]{FGS14b}.
We include a direct proof for completeness.

\begin{lemma}[A formula for the residual]
\label{lem:shm_arnoldi_res}
Let \(\spec(\Hessen_{m}) = \{\theta_{1}, \dotsc, \theta_{m}\}\), and assume that \(\theta_{j} \neq 0\) for \(j = 1, \dotsc, m\).
Then the residuals defined in~\eqref{eq:s_hm_arnoldi_res} satisfy
\begin{equation*}
\Mat{r}_{m}^{\sHmArn}(t)
=
\frac{1}
{(1 - t/\theta_{1})\dotsm(1 - t/\theta_{m})}
\Mat{r}_{m}^{\sHmArn}(0)
\eqcolon
\Mat{r}_{m}^{\sHmArn}(0)\eta_{m}(t).
\end{equation*}
\end{lemma}

\begin{proof}
Let \(\phi_{m}(z) = (z - \theta_{1})\dotsm(z - \theta_{m})\).
It is straightforward to verify that
\(p_{m - 1,t}^{\sHmArn}(z)
=
\bigl(1 - \phi_{m}(z)/\phi_{m}(t)\bigr)/(t - z)\)
is the polynomial interpolating \(h_{t}(z) = (t - z)^{-1}\) 
at \(\spec(\Hessen_{m})\).
Direct substitution gives
\(q_{m,t}^{\sHmArn}(z)
=
1 - (t - z)p_{m - 1,t}^{\sHmArn}(z)
=
\phi_{m}(z)/\phi_{m}(t)\).
Therefore,
\begin{equation*}
\Mat{r}_{m}^{\sHmArn}(t)
=
q_{m,t}^{\sHmArn}(\Mat{A})\Mat{b}
=
\frac{\phi_{m}(\Mat{A})}{\phi_{m}(t)}\Mat{b}
=
\frac{\phi_{m}(0)}{\phi_{m}(t)}
\frac{\phi_{m}(\Mat{A})}{\phi_{m}(0)}
\Mat{b}
=
\frac{\phi_{m}(0)}{\phi_{m}(t)}
\Mat{r}_{m}^{\sHmArn}(0).
\end{equation*}
Since
\(\phi_{m}(t)/\phi_{m}(0)
=
(1 - t/\theta_{1})\dotsm(1 - t/\theta_{m})\),
the stated formula follows.
\end{proof}

Suppose that \(\Mat{A}\) is positive real.
If the sketched harmonic Ritz values \(\spec(\Hessen_{m})\) 
lie in the open right half-plane \(\RightHalfPlane\), 
the modulus of the factor \(\phi_{m}(0)/\phi_{m}(t)\) 
decreases as \(t\) moves from \(0\) toward \(-\infty\).
Moreover, \(\Mat{r}_{m}^{\sHmArn}(0)\) 
is precisely the residual of the 
sGMRES approximant \(\Mat{x}_{m}^{\sHmArn}(0)\).
The following lemma provides a bound for 
\(\|\Mat{r}_{m}^{\sHmArn}(0)\|\).

\begin{lemma}[Bound on the residual]
\label{lem:sgmres_res_bound}
Let \(\Mat{A} \in \complex^{N \times N}\) be positive real, and let \(\Mat{b} \in \complex^{N}\).
Assume that the sketching matrix \(\Mat{S}\) 
satisfies~\eqref{eq:subspace_embedding_property} 
for some \(\varepsilon \in [0,1)\).
Let \(\Mat{x}_{m}^{\sHmArn}(t)\) be the sHmArnoldi approximation 
to the solution of the shifted linear system~\eqref{eq:shifted_linear_system}, 
and let \(\Mat{r}_{m}^{\sHmArn}(t)\) 
be the corresponding residual 
defined in~\eqref{eq:s_hm_arnoldi_res}.
Then
\begin{equation*}
\|\Mat{r}_{m}^{\sHmArn}(0)\|
\leq
C_{\varepsilon}\alpha_{m}\|\Mat{b}\|,
\end{equation*}
where
\(C_{\varepsilon}
=
\sqrt{(1 + \varepsilon)/(1 - \varepsilon)}\)
and
\(\alpha_{m}
=
(1 - \delta\delta^{\prime})^{m/2}
\in [0,1)\).
The constants \(\delta\) and \(\delta^{\prime}\) are defined 
by
\begin{equation}
\label{eq:delta_delta_prime}
\begin{aligned}
\delta
&=
\min
\bigl\{
|\Mat{v}^{\herm}\Mat{A}\Mat{v}|
:
\Mat{v} \in \complex^{N},
\ \|\Mat{v}\| = 1
\bigr\},
\\
\delta^{\prime}
&=
\min
\bigl\{
|\Mat{v}^{\herm}\Mat{A}^{-1}\Mat{v}|
:
\Mat{v} \in \complex^{N},
\ \|\Mat{v}\| = 1
\bigr\}.
\end{aligned}
\end{equation}
\end{lemma}

\begin{proof}
Let \(\Mat{x}_{m}^{\HmArn}(t)\) be the HmArnoldi approximation 
to the solution of the shifted linear system~\eqref{eq:shifted_linear_system}, 
and let \(\Mat{r}_{m}^{\HmArn}(t)\) 
be the corresponding residual.
By~\cite[Lemma~6.3]{FGS14b},
\(\|\Mat{r}_{m}^{\HmArn}(0)\|
\leq
\alpha_{m}\|\Mat{b}\|\).
Since the subspace embedding property holds 
for the relevant residual vectors, 
by the sketched residual minimization property of sGMRES, 
we have
\begin{equation*}
\begin{aligned}
\|\Mat{r}_{m}^{\sHmArn}(0)\|
& \leq
\frac{1}{\sqrt{1 - \varepsilon}}
\|\Mat{S}\Mat{r}_{m}^{\sHmArn}(0)\|
\leq
\frac{1}{\sqrt{1 - \varepsilon}}
\|\Mat{S}\Mat{r}_{m}^{\HmArn}(0)\|
\\
& \leq
\sqrt{\frac{1 + \varepsilon}{1 - \varepsilon}}
\|\Mat{r}_{m}^{\HmArn}(0)\|.
\end{aligned}
\end{equation*}
Combining the above estimate 
with
\(\|\Mat{r}_{m}^{\HmArn}(0)\|
\leq
\alpha_{m}\|\Mat{b}\|\)
completes the proof.
\end{proof}

For a Stieltjes function \(f\) 
of the form~\eqref{eq:stieltjes_func}, 
the sHmArnoldi approximation to \(f(\Mat{A})\Mat{b}\) 
can be written as
\begin{equation*}
\Mat{f}_{m}^{\sHmArn}
=
\int_{-\infty}^{0}
g(t)\Mat{x}_{m}^{\sHmArn}(t)
\diff t.    
\end{equation*}
Consequently,
\begin{equation*}
\begin{aligned}
f(\Mat{A})\Mat{b} - \Mat{f}_{m}^{\sHmArn}
&=
\int_{-\infty}^{0}
g(t)
\bigl(
    \Mat{x}(t) - \Mat{x}_{m}^{\sHmArn}(t)
\bigr)
\diff t
\\
&=
\int_{-\infty}^{0}
g(t)
(t\Mat{I} - \Mat{A})^{-1}
\bigl(
    \Mat{b}
    -
    (t\Mat{I} - \Mat{A})\Mat{x}_{m}^{\sHmArn}(t)
\bigr)
\diff t
\\
&=
\int_{-\infty}^{0}
g(t)
(t\Mat{I} - \Mat{A})^{-1}
\Mat{r}_{m}^{\sHmArn}(t)
\diff t
\\
&=
\int_{-\infty}^{0}
g(t)
(t\Mat{I} - \Mat{A})^{-1}
\eta_{m}(t)\Mat{r}_{m}^{\sHmArn}(0)
\diff t
\\
&=
\errfunc(\Mat{A})\Mat{r}_{m}^{\sHmArn}(0),
\end{aligned}
\end{equation*}
where
\(\errfunc(z)
=
\int_{-\infty}^{0}
g(t)\eta_{m}(t)(t - z)^{-1}
\diff t\).
Combining the above discussion, 
we obtain the convergence of the restarted sHmArnoldi method 
for computing \(f(\Mat{A})\Mat{b}\) 
when \(\Mat{A}\) is positive real 
and \(f\) is a Stieltjes function.

\begin{theorem}[Convergence of the restarted sHmArnoldi method]
\label{thm:sketched_harmonic_convergence}
Let \(\Mat{A} \in \complex^{N \times N}\) be positive real, 
let \(\Mat{b} \in \complex^{N}\), 
and let \(f\) be a Stieltjes function 
of the form~\eqref{eq:stieltjes_func}.
Fix an integer \(k \geq 0\).
Assume that the sketching matrix \(\Mat{S}\),
reused in every restart cycle, satisfies the norm inequalities
in~\eqref{eq:subspace_embedding_property} with a common
\(\varepsilon \in [0,1)\) for every \(0 \leq \ell \leq k\)
and every
\(\Mat{v} \in \Kry_{m + 1}(\Mat{A},\basis_{1}^{[\ell]})\).
Here, \(\basis_{1}^{[\ell]}\) is the first basis vector
in the \(\ell\)-th cycle.
Assume also that the sketched harmonic Ritz values 
satisfy
\(\theta_{j}^{[\ell]} \in \RightHalfPlane\)
for \(0 \leq \ell \leq k\) and \(1 \leq j \leq m\).
Then the restarted sHmArnoldi approximation computed by
Algorithm~\ref{alg:restart_skrylov} satisfies
\begin{equation}
\label{eq:sketched_harmonic_bound}
\bigl\|
f(\Mat{A})\Mat{b}
-
\Mat{f}^{\sHmArn,[k]}
\bigr\|_{\Mat{A}^{\herm}\Mat{A}}
\leq
C\|\Mat{b}\|
\bigl(C_{\varepsilon}\alpha_{m}\bigr)^{k + 1},
\end{equation}
where
\(C = \sqrt{\nu_{\max}}f(\rho\nu_{\max})\), and
\(\alpha_{m}\) and \(C_{\varepsilon}\) 
are defined in Lemma~\ref{lem:sgmres_res_bound}.
Here,
\begin{equation}
\label{eq:nu_max_rho}
\begin{aligned}
\nu_{\max}
&=
\max
\bigl\{
\Mat{v}^{\herm}\Mat{A}^{\herm}\Mat{A}\Mat{v}
:
\Mat{v} \in \complex^{N},
\ \|\Mat{v}\| = 1
\bigr\}
=
\|\Mat{A}\|^{2},
\\
\rho
&=
\min
\bigl\{
\Re\bigl(
    \Mat{v}^{\herm}\Mat{A}^{-1}\Mat{v}
\bigr)
:
\Mat{v} \in \complex^{N},
\ \|\Mat{v}\| = 1
\bigr\}.
\end{aligned}
\end{equation}
Since \(\Mat{A}\) is positive real, 
we have \(\rho > 0\).
Thus, \(\rho\nu_{\max} > 0\), 
and the representation~\eqref{eq:stieltjes_func} 
with \(g\) nonpositive implies that \(f(\rho\nu_{\max}) > 0\).
\end{theorem}

\begin{proof}
By Lemma~\ref{lem:shm_arnoldi_res}, 
the residual in the \(\ell\)-th cycle satisfies
\(\Mat{r}_{m}^{\sHmArn,[\ell]}(t)
=
\Mat{r}_{m}^{\sHmArn,[\ell]}(0)\eta_{m}^{[\ell]}(t)\).
Following the argument 
in Corollary~\ref{cor:err_arnoldi_like_approx_k_restart}, 
the remaining error after \(k\) restart cycles
can be written as
\begin{equation*}
f(\Mat{A})\Mat{b} - \Mat{f}^{\sHmArn,[k]}
=
\biggl(
\int_{-\infty}^{0}
g(t)
\prod_{\ell = 0}^{k}\eta_{m}^{[\ell]}(t)
(t\Mat{I} - \Mat{A})^{-1}
\diff t
\biggr)
\Mat{r}_{m}^{\sHmArn,[k]}(0).
\end{equation*}
Since \(\Re(\theta_{j}^{[\ell]}) > 0\), 
we have
\(\Re(-t/\theta_{j}^{[\ell]}) \geq 0\) for all \(t \leq 0\).
Hence, \(|1 - t/\theta_{j}^{[\ell]}| \geq 1\), and
\begin{equation*}
|\eta_{m}^{[\ell]}(t)|
=
\prod_{j = 1}^{m}
\frac{1}{|1 - t/\theta_{j}^{[\ell]}|}
\leq
1,
\qquad
0 \leq \ell \leq k,
\quad
t \leq 0.
\end{equation*}
Following the proof of~\cite[Theorem~6.5]{FGS14b} 
and using
\(|\eta_{m}^{[\ell]}(t)| \leq 1\), 
we obtain
\begin{equation*}
\begin{aligned}
\bigl\|
f(\Mat{A})\Mat{b}
-
\Mat{f}^{\sHmArn,[k]}
\bigr\|_{\Mat{A}^{\herm}\Mat{A}}
&\leq
\|\Mat{r}_{m}^{\sHmArn,[k]}(0)\|
\int_{-\infty}^{0}
g(t)
\frac{\sqrt{\nu_{\max}}}{t - \rho\nu_{\max}}
\diff t
\\
&=
\|\Mat{r}_{m}^{\sHmArn,[k]}(0)\|
\sqrt{\nu_{\max}}
f(\rho\nu_{\max}).
\end{aligned}
\end{equation*}
Here, we used the fact that \(g\) is nonpositive 
in the Stieltjes representation~\eqref{eq:stieltjes_func}.
Finally, Lemma~\ref{lem:sgmres_res_bound} 
gives
\(\|\Mat{r}_{m}^{\sHmArn,[k]}(0)\|
\leq
C_{\varepsilon}\alpha_{m}
\|\Mat{r}_{m}^{\sHmArn,[k - 1]}(0)\|\).
Applying this inequality recursively and defining
\(C = \sqrt{\nu_{\max}}f(\rho\nu_{\max})\)
completes the proof.
\end{proof}

As explained in Section~\ref{subsec:hm_arnoldi}, 
the non-sketched adaptive variants inherit convergence 
for Stieltjes functions 
and every restart length \(m \geq 1\), 
with Hermitian positive definite matrices 
for Arnoldi 
and positive real matrices 
for HmArnoldi~\cite[Theorems~4.3 and~6.5]{FGS14b}.
In the latter setting, Theorem~\ref{thm:sketched_harmonic_convergence} 
adds embedding and 
sketched harmonic Ritz-value assumptions 
and replaces the contraction factor 
\(\alpha_m\) by \(C_{\varepsilon}\alpha_m\).

The factor \(C_{\varepsilon}\) 
compares sketched and Euclidean residual norms 
and depends on embedding distortion, 
independently of basis orthonormality.
If \(\varepsilon = 0\) on the relevant subspace in every cycle, sketching preserves inner products, 
recovering Euclidean harmonic orthogonality, 
the HmArnoldi approximants, 
and their bound with \(C_{\varepsilon}=1\).

In exact arithmetic and neglecting quadrature errors, 
geometric convergence follows 
if the embedding and Ritz-value assumptions 
hold in every cycle with a common \(\varepsilon\) 
satisfying \(C_{\varepsilon}\alpha_{m} < 1\).
Failure of this sufficient inequality 
does not imply divergence, 
and the bound is not intended to 
predict the observed rate.

\subsection{Sketched harmonic Ritz values}
\label{subsec:s_hm_ritz_values}

The convergence result 
in Theorem~\ref{thm:sketched_harmonic_convergence} 
assumes that all sketched harmonic Ritz values 
lie in the open right half-plane.
This subsection provides a sufficient condition 
under which this assumption holds.
The argument relies only on the orthogonality condition 
produced by the sHm-Arnoldi process, 
namely,
\(\basis_{m + 1} \perp_{\Mat{S}} \Mat{A}\Basis_{m}\).

Denote \(\ABasis_{m} = \Mat{A}\Basis_{m}\).
Since \(\basis_{m + 1} \perp_{\Mat{S}} \ABasis_{m}\), 
multiplying the Arnoldi-like decomposition~\eqref{eq:arnoldi_like_decomp} 
from the left by \((\Mat{S}\ABasis_{m})^{\herm}\Mat{S}\) 
gives
\begin{equation*}
(\Mat{S}\ABasis_{m})^{\herm}
(\Mat{S}\ABasis_{m})
=
(\Mat{S}\ABasis_{m})^{\herm}
(\Mat{S}\Basis_{m})
\Hessen_{m}.
\end{equation*}
Therefore,
\(\Hessen_{m} = \Mat{G}_{m}^{-1}\Mat{F}_{m}\), where
\begin{equation*}
\Mat{G}_{m}
=
(\Mat{S}\ABasis_{m})^{\herm}(\Mat{S}\Basis_{m})
\quad \text{and} \quad 
\Mat{F}_{m}
=
(\Mat{S}\ABasis_{m})^{\herm}(\Mat{S}\ABasis_{m}),
\end{equation*}
provided that \(\Mat{G}_{m}\) is nonsingular.
Thus, the sketched harmonic Ritz values 
are the generalized eigenvalues 
of the matrix pencil \((\Mat{F}_{m},\Mat{G}_{m})\).

Let \(\theta \in \spec(\Hessen_{m})\), 
and let \(\Mat{y} \in \complex^{m}\) 
be a corresponding nonzero eigenvector.
Set \(\Mat{x} = \ABasis_{m}\Mat{y}\).
Then \(\Mat{x} \neq \Mat{0}\) and
\(\Mat{F}_{m}\Mat{y} = \Mat{G}_{m}\Mat{y}\theta\).
This gives
\begin{equation*}
\theta
=
\frac{\Mat{y}^{\herm}\Mat{F}_{m}\Mat{y}}{\Mat{y}^{\herm}\Mat{G}_{m}\Mat{y}}
=
\frac{\|\Mat{S}\Mat{x}\|^{2}}{(\Mat{S}\Mat{x})^{\herm}(\Mat{S}\Mat{A}^{-1}\Mat{x})}.
\end{equation*}
Equivalently,
\begin{equation} \label{eq:recip_s_hm_ritz_value}
\frac{1}{\theta}
=
\frac{(\Mat{S}\Mat{x})^{\herm}(\Mat{S}\Mat{A}^{-1}\Mat{x})}{\|\Mat{S}\Mat{x}\|^{2}}.
\end{equation}

Taking 
\(\Mat{u} = \Mat{A}^{-1}\Mat{x} = \Basis_{m}\Mat{y}\) 
and 
\(\Mat{v} = \Mat{x} = \Mat{A}\Basis_{m}\Mat{y}\) 
in the inner-product form of 
the subspace embedding property~\eqref{eq:subspace_embedding_property_inner_product}, 
we obtain
\begin{equation*}
\bigl|
(\Mat{S}\Mat{x})^{\herm}
(\Mat{S}\Mat{A}^{-1}\Mat{x})
-
\Mat{x}^{\herm}\Mat{A}^{-1}\Mat{x}
\bigr|
\leq
\varepsilon
\|\Mat{A}^{-1}\Mat{x}\|
\|\Mat{x}\|
\leq
\varepsilon
\|\Mat{A}^{-1}\|
\|\Mat{x}\|^{2}.
\end{equation*}
Combining this estimate 
with the subspace embedding property~\eqref{eq:subspace_embedding_property} 
and the identity~\eqref{eq:recip_s_hm_ritz_value}, 
we conclude that the reciprocal 
of every sketched harmonic Ritz value 
satisfies
\begin{equation*}
\frac{1}{\theta}
\in
\biggl\{
\frac{z_{1} + z_{2}}{\tau}
:
z_{1} \in \numrange(\Mat{A}^{-1}),\
|z_{2}| \leq \varepsilon\|\Mat{A}^{-1}\|,\
1 - \varepsilon \leq \tau \leq 1 + \varepsilon
\biggr\}.
\end{equation*}

Since \(\Mat{A}\) is positive real, 
\(\Mat{A}^{-1}\) is also positive real.
Let \(\rho\) be defined as in~\eqref{eq:nu_max_rho}.
If
\(\varepsilon\|\Mat{A}^{-1}\| < \rho\),
then, for any sketched harmonic Ritz value \(\theta\), 
the preceding inclusion gives
\begin{equation*}
\Re\biggl(\frac{1}{\theta}\biggr)
\geq
\frac{\rho - \varepsilon\|\Mat{A}^{-1}\|}{1 + \varepsilon}
>
0.
\end{equation*}
Since \(\Re(1/\theta) > 0\) 
is equivalent to \(\Re(\theta) > 0\), 
all sketched harmonic Ritz values 
lie in the open right half-plane.
Moreover, under the condition
\(\varepsilon\|\Mat{A}^{-1}\| < \rho\), 
for every nonzero
\(\Mat{y} \in \complex^{m}\), 
with
\(\Mat{x} = \Mat{A}\Basis_{m}\Mat{y}\), 
we have
\begin{equation*}
\Re\bigl(\Mat{y}^{\herm}\Mat{G}_{m}\Mat{y}\bigr)
=
\Re\bigl((\Mat{S}\Mat{x})^{\herm}
(\Mat{S}\Mat{A}^{-1}\Mat{x})\bigr)
\geq
\bigl(\rho-\varepsilon\|\Mat{A}^{-1}\|\bigr)\|\Mat{x}\|^{2}
> 0.
\end{equation*}
Hence, \(\Mat{G}_{m}\) is nonsingular, 
and the condition
\(\varepsilon\|\Mat{A}^{-1}\| < \rho\)
is sufficient for the sketched harmonic Ritz values generated by the sHm-Arnoldi process to lie in \(\RightHalfPlane\), 
as required in Theorem~\ref{thm:sketched_harmonic_convergence}.

Combining this condition with \(C_{\varepsilon} \alpha_{m}<1\) 
gives the joint sufficient restriction
\begin{equation*}
0 \leq \varepsilon <
\min\left\{
\frac{1-\alpha_{m}^{2}}{1+\alpha_{m}^{2}},
\frac{\rho}{\|\Mat{A}^{-1}\|}
\right\},
\end{equation*}
provided that~\eqref{eq:subspace_embedding_property} 
holds in every restart cycle with this common \(\varepsilon\).
The required distortion depends on 
\(\Mat{A}\) and, 
through \(\alpha_{m}\), 
on the restart length \(m\).
Choosing \(s\) balances 
this problem-dependent accuracy requirement 
against the cost of sketching.

\section{Numerical experiments}
\label{sec:numerical_experiments}

In this section, 
we present numerical experiments 
to demonstrate the stability and efficiency 
of the proposed sketch-and-restart algorithms.
All algorithms are implemented in MATLAB R2023b, 
and all experiments are performed on a server equipped 
with two Intel Xeon Gold 6226R CPUs running at 2.90 GHz 
and 1 TB of RAM.
The code for reproducing the numerical experiments is available at \url{https://github.com/JingyuLiuMath/funm_quad_sketch_and_restart}.

The two matrix functions used in the experiments are
\begin{equation*}
f_{1}(z)
=
z^{-1/2}
=
\frac{-1}{\pi}
\int_{-\infty}^{0}
\frac{(-t)^{-1 / 2}}{t - z} \diff t
\end{equation*}
and
\begin{equation*}
f_{2}(z)
=
\frac{\log(1 + z)}{z}
=
\int_{-\infty}^{-1}
\frac{t^{-1}}{t - z} \diff t.
\end{equation*}
We use the sparse sign matrices described in
Section~\ref{subsec:skrylov},
with \(\zeta = 8\) following~\cite{Martinsson_Tropp_2020}.
For this fixed sparsity, applying \(\Mat{S}\) to a vector
costs \(\bigO(N)\) operations.

The variants of Algorithm~\ref{alg:restart_skrylov}
based on the s-Arnoldi and sHm-Arnoldi processes
are denoted by sArn and sHmArn, respectively;
both use a fixed Krylov subspace dimension \(m\)
and sketching dimension \(s = 4m\).
For comparison, we use the quadrature-based restarted Arnoldi
method of~\cite{Frommer_Guettel_Schweitzer_2014}
and its harmonic counterpart, denoted by Arn and HmArn.
These use the standard Arnoldi process~\cite{Saad_2003}
and a deterministic variant of
Algorithm~\ref{alg:sketched_harmonic_arnoldi}, respectively.

For Algorithm~\ref{alg:restart_krylov_adaptive},
we set \(t \in \{2,1,0\}\),
\(\tau_{\condcontrol} = 10^{6}\),
\(s_{0} = 30\), and \(\eta = 4\).
Each new vector is orthogonalized against at most
the \(t\) most recent basis vectors;
\(t = 0\) omits these steps but retains normalization
and the final rank-\(1\) update.
This update enforces
\(\basis_{m+1} \perp \Basis_{m}\)
or
\(\basis_{m+1} \perp \Mat{A}\Basis_{m}\),
giving the methods adaArn-\(t\) and adaHmArn-\(t\),
respectively; see Section~\ref{subsec:rank1_update}.

All methods use quadrature tolerance
\(\tau_{\quadature} = 10^{-9}\),
stopping tolerance \(\tau = 10^{-8}\),
and at most \(n_{\quadature,\max} = 1024\) quadrature nodes.
The stopping test compares the norm of the correction
with \(\tau\) times that of the current approximation.

We report the relative error
\(\|\Mat{f}^{[k]}-\Mat{f}_{\myref}\|_{2}/
\|\Mat{f}_{\myref}\|_{2}\),
where the reference vector \(\Mat{f}_{\myref}\)
is supplied with the MATLAB code
accompanying~\cite{Guettel_Schweitzer_2023}
for Section~\ref{subsec:qcd},
computed using \texttt{sqrtm} applied to the full matrix
for Section~\ref{subsec:graph_laplacian},
or obtained from a known spectral decomposition of \(\Mat{A}\)
for Section~\ref{subsec:conv_diff}.
The quantities \(n_{\iter}\), \(n_{\matvec}\), and \(m_{\actual}\)
denote the number of restart cycles,
the total number of matrix-vector products,
and the average actual Krylov subspace dimension
for each cycle,
respectively.

To improve convergence, we also consider thick restarting,
which retains a small approximate invariant subspace across cycles
to preserve spectral information and can improve convergence
through approximate deflation~\cite{Eiermann_Ernst_Guettel_2011}.
For the adaptive Arnoldi methods, we retain lifted Schur vectors
associated with Ritz values having the smallest absolute real parts.
The target retention dimension is \(5\), increased when necessary
to preserve a conjugate pair.
Our current implementation of the harmonic Arnoldi processes
does not support thick restarting, so this comparison
is restricted to the Arnoldi variants.

We run each configuration three times with different random seeds
and select the run with the median total runtime.
All results reported in the tables and figures
are taken from this selected run.
All timings are reported in seconds.

\subsection{Lattice quantum chromodynamics example}
\label{subsec:qcd}

This example arises from lattice quantum chromodynamics~(QCD), 
where one needs to compute \(\sign(\Mat{Q})\)
on a vector for a large sparse non-Hermitian matrix 
\(\Mat{Q} \in \complex^{N \times N}\).
Since forming \(\sign(\Mat{Q})\) explicitly is infeasible, 
its action on a vector is computed as
\begin{equation*}
\sign(\Mat{Q})\Mat{v}
=
(\Mat{Q}^{2})^{-1/2} \Mat{Q}\Mat{v}
=
f_{1}(\Mat{A}) \Mat{b},
\end{equation*}
where \(\Mat{A} = \Mat{Q}^{2}\) can be accessed only through matrix-vector products.
In our experiment, \(N = 49152\),
\(\Mat{v} = \Mat{e}_{1}/\|\Mat{Q}\Mat{e}_{1}\|\),
so that
\(\Mat{b} = \Mat{Q}\Mat{v}
= \Mat{Q}\Mat{e}_{1}/\|\Mat{Q}\Mat{e}_{1}\|\).
We choose 
\(m = 150\) for fixed-dimension methods 
and \(m_{\max} = 150\) for adaptive methods.

\begin{table}[htbp]
\centering
\caption{Relative error, runtime in seconds, number of restart cycles \(n_{\mathrm{iter}}\), number of matrix-vector products \(n_{\mathrm{matvec}}\), and average actual Krylov subspace dimension \(m_{\mathrm{act}}\) per restart cycle for the QCD example. The column ``speedup'' reports the percentage runtime reduction \((t_{\mathrm{base}}-t_{\mathrm{curr}})/t_{\mathrm{base}}\times100\%\), where the baseline is Arn (left group) or HmArn (right group). For fixed-dimension methods, \(m=150\) and \(s=4m=600\). For adaptive methods, \(m_{\max}=150\) and \(\eta=4\). The row corresponding to the fastest method in each group is highlighted in bold.}
\label{tab:qcd-invsqrt-sfactor-4}
\resizebox{\textwidth}{!}{%
\begin{tabular}{@{}lrrrrrr@{\quad}|@{\quad}lrrrrrr@{}}
\toprule
method & error & time & \(n_{\mathrm{iter}}\) & \(n_{\mathrm{matvec}}\) & \(m_{\mathrm{act}}\) & speedup & method & error & time & \(n_{\mathrm{iter}}\) & \(n_{\mathrm{matvec}}\) & \(m_{\mathrm{act}}\) & speedup \\
\midrule
Arn & 1.5e-12 & 18.4 & 4 & 600 & 150.0 & \(0.0\%\) & HmArn & 5.6e-12 & 17.7 & 4 & 600 & 150.0 & \(0.0\%\) \\
sArn & 4.4e-12 & 16.2 & 4 & 600 & 150.0 & \(12.2\%\) & sHmArn & 7.0e-12 & 16.7 & 4 & 600 & 150.0 & \(5.2\%\) \\
{\bfseries\boldmath adaArn-2} & {\bfseries\boldmath 2.5e-10} & {\bfseries\boldmath 12.9} & {\bfseries\boldmath 5} & {\bfseries\boldmath 514} & {\bfseries\boldmath 102.8} & {\bfseries\boldmath \(29.8\%\)} & adaHmArn-2 & 2.2e-11 & 15.4 & 6 & 609 & 101.5 & \(13.1\%\) \\
adaArn-1 & 4.0e-8 & 15.2 & 46 & 788 & 17.1 & \(17.5\%\) & {\bfseries\boldmath adaHmArn-1} & {\bfseries\boldmath 4.6e-8} & {\bfseries\boldmath 9.8} & {\bfseries\boldmath 26} & {\bfseries\boldmath 483} & {\bfseries\boldmath 18.6} & {\bfseries\boldmath \(44.5\%\)} \\
adaArn-0 & 1.4e-7 & 25.8 & 141 & 1269 & 9.0 & \(-40.0\%\) & adaHmArn-0 & 1.5e-7 & 27.2 & 140 & 1260 & 9.0 & \(-54.0\%\) \\
\bottomrule
\end{tabular}%
}
\end{table}

\begin{figure}[tbhp]
\centering
\includegraphics[width=.49\textwidth]{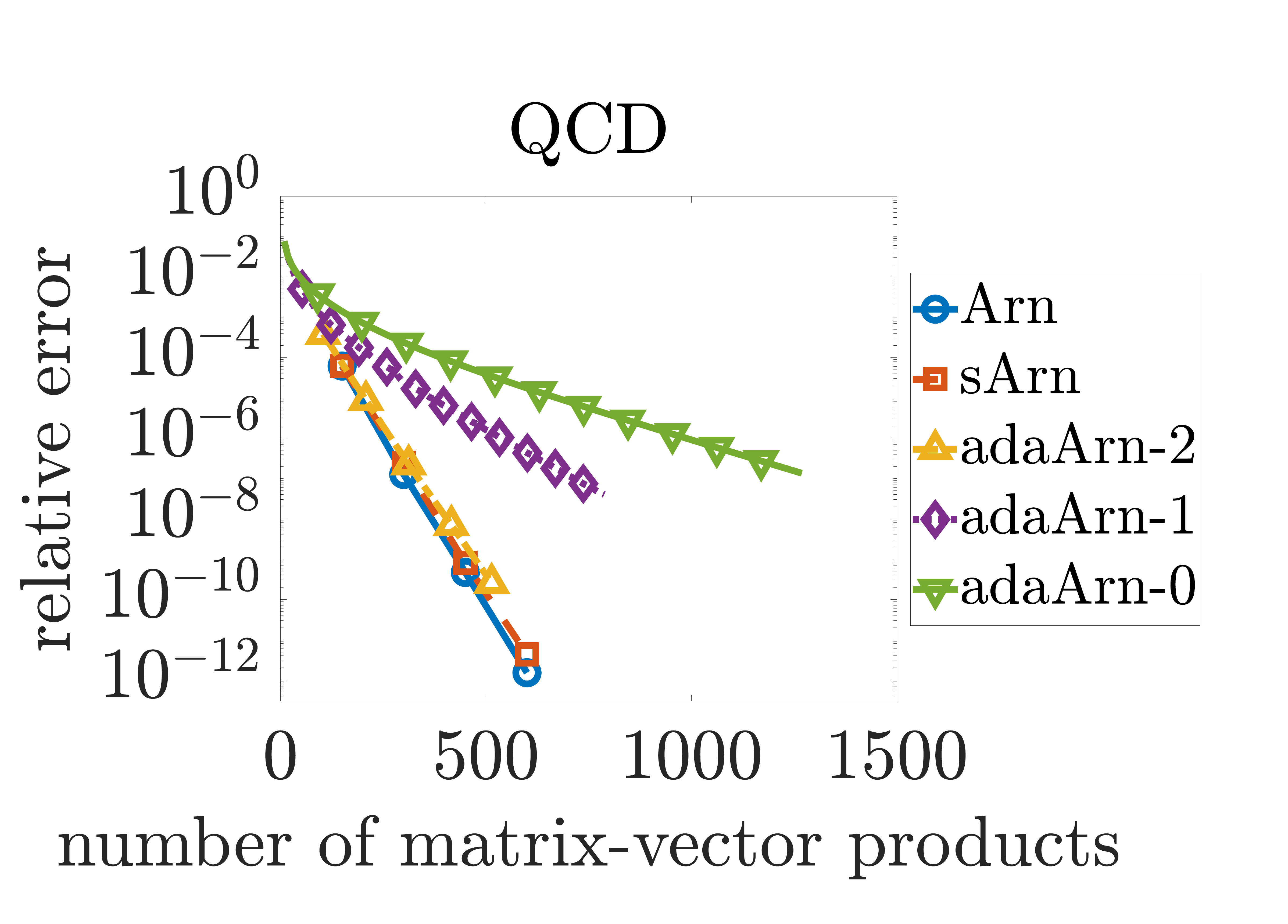}
\includegraphics[width=.49\textwidth]{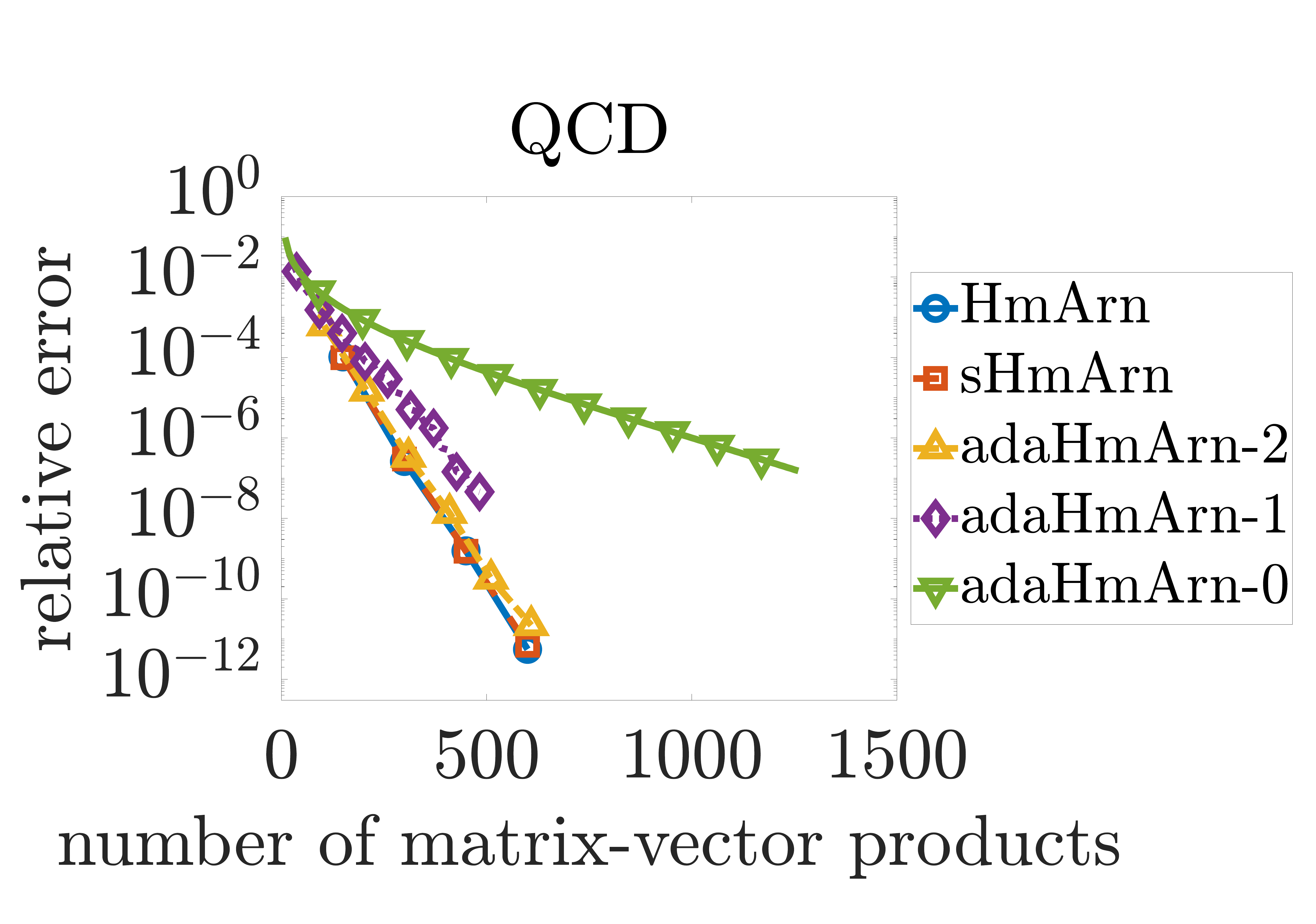}
\caption{Relative error versus the number of matrix-vector products for the QCD example. For fixed-dimension methods, \(m=150\) and \(s=4m=600\). For adaptive methods, \(m_{\max}=150\) and \(\eta=4\).}
\label{fig:qcd-invsqrt-sfactor-4-err}
\end{figure}

\begin{figure}[tbhp]
\centering
\includegraphics[width=.49\textwidth]{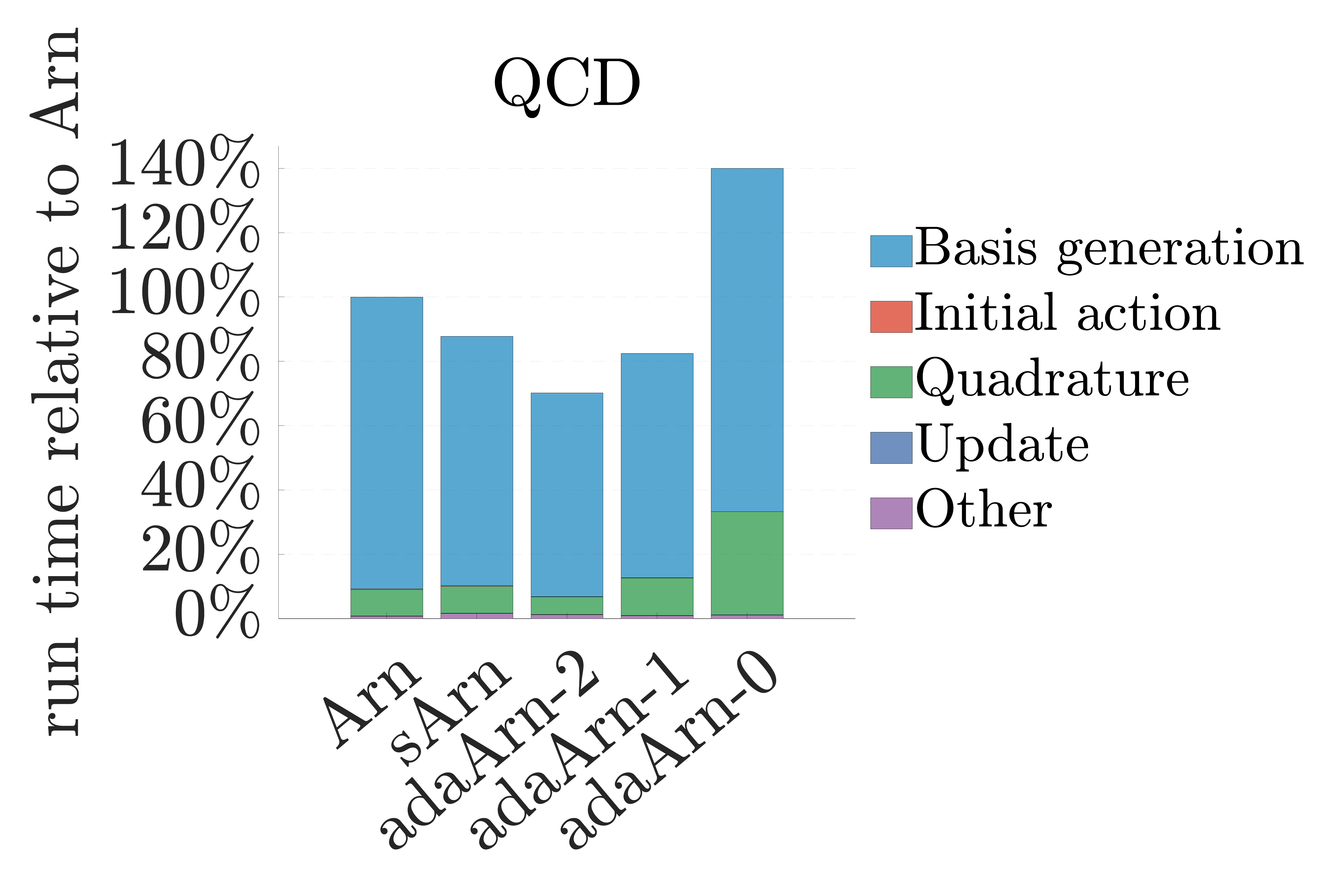}
\includegraphics[width=.49\textwidth]{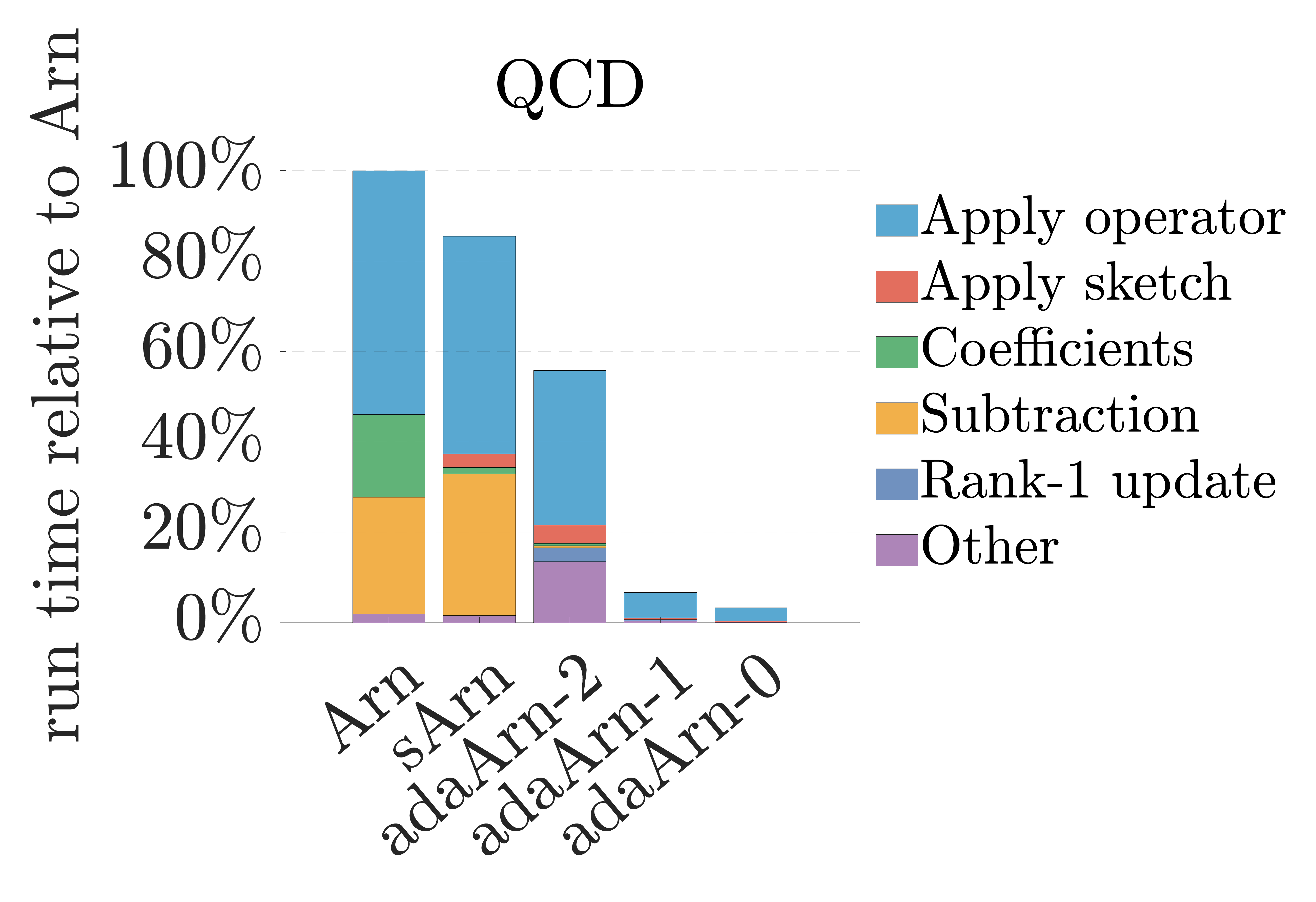}
\caption{Runtime breakdown for the QCD example.
The left panel shows the total runtime, normalized by that of Arn.
The right panel shows the average runtime per Arnoldi call,
normalized by the corresponding average for Arn.
For fixed-dimension methods, \(m=150\) and \(s=4m=600\).
For adaptive methods, \(m_{\max}=150\) and \(\eta=4\).}
\label{fig:qcd-invsqrt-sfactor-4-ratios}
\end{figure}

Table~\ref{tab:qcd-invsqrt-sfactor-4} and
Figure~\ref{fig:qcd-invsqrt-sfactor-4-err}
show that adaArn-2 and adaHmArn-1 are the fastest methods
in the Arnoldi and harmonic Arnoldi groups, respectively.
Their runtimes are \(12.9\) and \(9.8\) seconds,
corresponding to runtime reductions of \(29.8\%\) and \(44.5\%\)
relative to Arn and HmArn, respectively.
They also require the fewest matrix-vector products
in their respective groups: \(514\) and \(483\).
Their final relative errors are
\(2.5\times10^{-10}\) and \(4.6\times10^{-8}\), respectively.

The methods sArn and sHmArn closely follow the convergence
of Arn and HmArn, respectively.
All four fixed-dimension methods require \(4\) restart cycles
and \(600\) matrix-vector products.
Each cycle generates \(150\) basis vectors before
the stopping test is applied.
Consequently, these methods require relatively few cycles
but attain final relative errors of order \(10^{-12}\),
far below the stopping tolerance \(10^{-8}\).

Figure~\ref{fig:qcd-invsqrt-sfactor-4-ratios}
shows that basis generation accounts for most of the total runtime.
The harmonic Arnoldi variants exhibit a similar runtime breakdown,
so only the Arnoldi variants are shown.  In the runtime breakdown, ``Coefficients'' refers to the computation of the inner products between Krylov vectors, and ``Subtraction'' refers to the vector-vector updates incurred during the basis construction. 
For sArn, applying the sketch and computing
the orthogonalization coefficients have nearly negligible costs.
However, matrix-vector products with \(\Mat{A}\)
still account for a large fraction of the basis-generation time,
and updating the full-dimensional basis vectors
remains necessary.
These costs limit the overall savings from sketching:
sArn and sHmArn reduce the runtime by only
\(12.2\%\) and \(5.2\%\)
relative to Arn and HmArn, respectively.

Decreasing \(t\) reduces the average actual Krylov dimension
\(m_{\mathrm{act}}\) and increases the number of restart cycles
\(n_{\mathrm{iter}}\), but does not necessarily increase the runtime.
The smaller subspaces make basis generation cheaper per cycle,
which can compensate for the additional cycles.
For example, reducing \(t\) from \(2\) to \(1\)
in the harmonic group increases the number of cycles
from \(6\) to \(26\), but reduces the runtime
from \(15.4\) to \(9.8\) seconds.
Further reducing \(t\) to \(0\) gives an average dimension of \(9\),
but approximately \(140\) cycles are needed,
making both adaptive methods slower than their respective baselines.
Thus, the runtime depends on the balance between
the cost per cycle and the number of cycles.
The stopping test monitors the relative correction,
so the final relative error need not be below \(10^{-8}\).

\subsection{Fractional graph Laplacian example}
\label{subsec:graph_laplacian}

In this example, 
\(\Mat{W}\) is the nonsymmetric adjacency matrix
of size \(N = 62586\) corresponding to the directed graph
\texttt{p2p\_Gnutella31} 
from the SuiteSparse Matrix Collection~\cite{Davis_Hu_2011}.
We define the in-degree graph Laplacian by
\(\Mat{A} = \Mat{D} - \Mat{W}\)
and aim to compute \(\Mat{A}^{1 / 2} \Mat{v}\) 
for a given vector \(\Mat{v} \in \complex^{N}\).
The matrix \(\Mat{A}\) is singular, and its spectrum lies in the closed right half-plane.
We interpret the inverse square root as acting on the invariant subspace 
associated with the nonzero spectrum and compute
\begin{equation*}
\Mat{A}^{1/2}\Mat{v}
=
\Mat{A}^{-1/2}\Mat{A}\Mat{v}
=
f_{1}(\Mat{A})\Mat{b},
\end{equation*}
where \(\Mat{b} = \Mat{A}\Mat{v}\).
Since \(\Mat{A}\Mat{v} \in \range(\Mat{A})\), 
the Krylov subspace methods considered here can be applied to this inverse-square-root action.
In our experiments, 
both the real and imaginary parts of \(\Mat{v}\) are drawn from the standard normal distribution.
We set \(m = 150\) in this example.

\begin{table}[htbp]
\centering
\caption{Relative error, runtime in seconds, number of restart cycles \(n_{\mathrm{iter}}\), number of matrix-vector products \(n_{\mathrm{matvec}}\), and average actual Krylov subspace dimension \(m_{\mathrm{act}}\) per restart cycle for the graph Laplacian example. The column ``speedup'' reports the percentage runtime reduction \((t_{\mathrm{base}}-t_{\mathrm{curr}})/t_{\mathrm{base}}\times100\%\), where the baseline is Arn (left group) or HmArn (right group). For fixed-dimension methods, \(m=150\) and \(s=4m=600\). For adaptive methods, \(m_{\max}=150\) and \(\eta=4\). The row corresponding to the fastest method in each group is highlighted in bold.}
\label{tab:graph-laplacian-invsqrt-sfactor-4}
\resizebox{\textwidth}{!}{%
\begin{tabular}{@{}lrrrrrr@{\quad}|@{\quad}lrrrrrr@{}}
\toprule
method & error & time & \(n_{\mathrm{iter}}\) & \(n_{\mathrm{matvec}}\) & \(m_{\mathrm{act}}\) & speedup & method & error & time & \(n_{\mathrm{iter}}\) & \(n_{\mathrm{matvec}}\) & \(m_{\mathrm{act}}\) & speedup \\
\midrule
Arn & 5.4e-12 & 12.5 & 5 & 750 & 150.0 & \(0.0\%\) & HmArn & 1.3e-11 & 13.1 & 5 & 750 & 150.0 & \(0.0\%\) \\
{\bfseries\boldmath sArn} & {\bfseries\boldmath 1.4e-11} & {\bfseries\boldmath 8.9} & {\bfseries\boldmath 5} & {\bfseries\boldmath 750} & {\bfseries\boldmath 150.0} & {\bfseries\boldmath \(28.6\%\)} & {\bfseries\boldmath sHmArn} & {\bfseries\boldmath 1.5e-11} & {\bfseries\boldmath 9.2} & {\bfseries\boldmath 5} & {\bfseries\boldmath 750} & {\bfseries\boldmath 150.0} & {\bfseries\boldmath \(30.0\%\)} \\
adaArn-2 & 2.3e-7 & 39.5 & 325 & 4864 & 15.0 & \(-216.1\%\) & adaHmArn-2 & 3.6e-7 & 44.4 & 299 & 4784 & 16.0 & \(-238.1\%\) \\
adaArn-1 & 3.5e-7 & 36.0 & 365 & 4743 & 13.0 & \(-188.1\%\) & adaHmArn-1 & 3.2e-7 & 39.4 & 366 & 4758 & 13.0 & \(-200.4\%\) \\
adaArn-0 & 4.0e-7 & 52.8 & 638 & 5740 & 9.0 & \(-322.4\%\) & adaHmArn-0 & 7.6e-7 & 39.3 & 480 & 4320 & 9.0 & \(-199.3\%\) \\
\bottomrule
\end{tabular}%
}
\end{table}

\begin{figure}[tbhp]
\centering
\includegraphics[width=.49\textwidth]{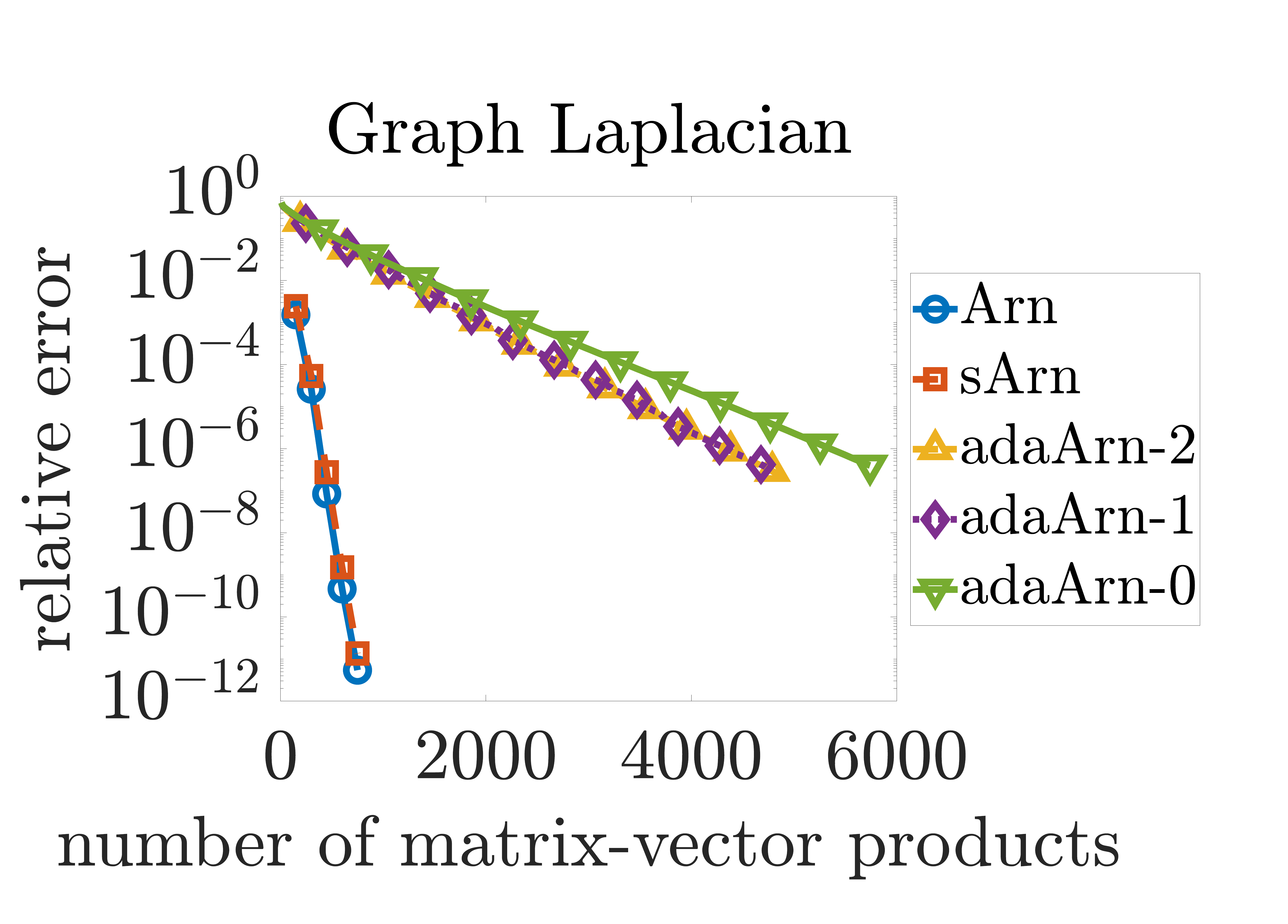}
\includegraphics[width=.49\textwidth]{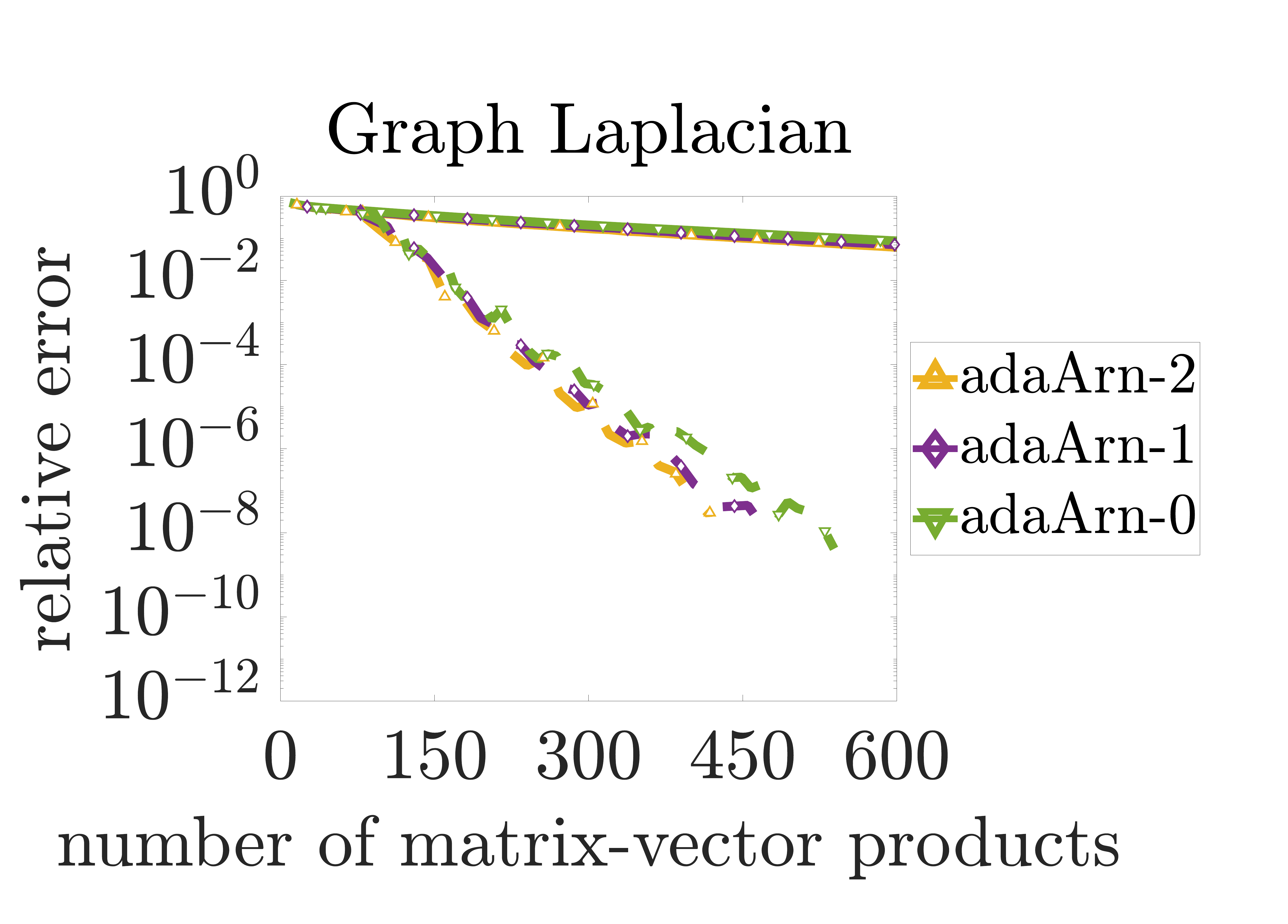}
\caption{Convergence for the fractional graph Laplacian example.
Left: relative error versus the number of matrix-vector products
without thick restarting.
Right: adaptive Arnoldi methods without thick restarting
(solid curves) and with thick restarting (dashed curves),
shown over a smaller range of matrix-vector products.
Fixed-dimension methods use \(m=150\) and \(s=4m=600\);
adaptive methods use \(m_{\max}=150\) and \(\eta=4\).
The target retention dimension for thick restarting is \(5\),
increased when necessary to preserve a conjugate pair.}
\label{fig:graph-laplacian-invsqrt-sfactor-4-err-thick}
\end{figure}

\begin{figure}[tbhp]
\centering
\includegraphics[width=.49\textwidth]{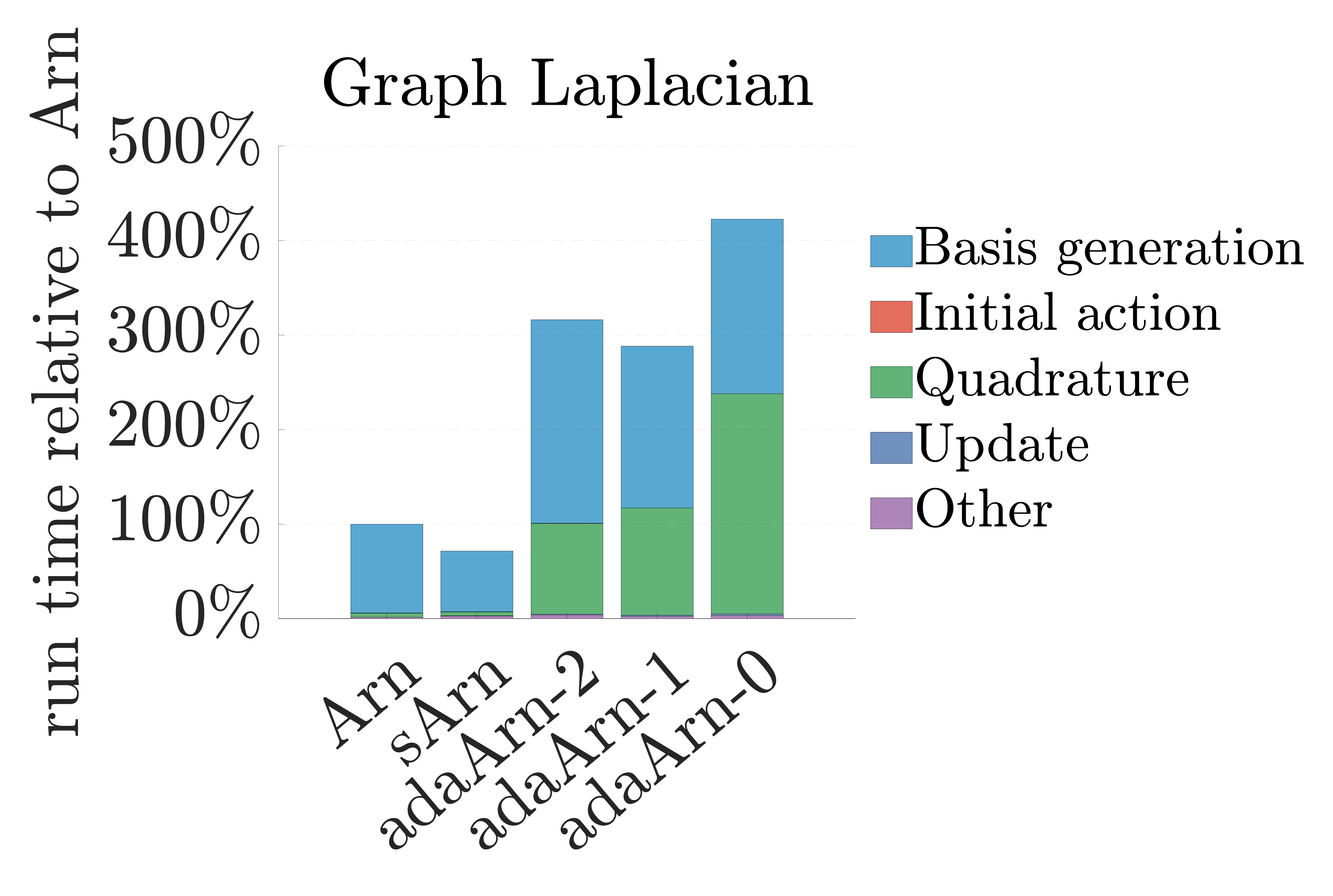}
\includegraphics[width=.49\textwidth]{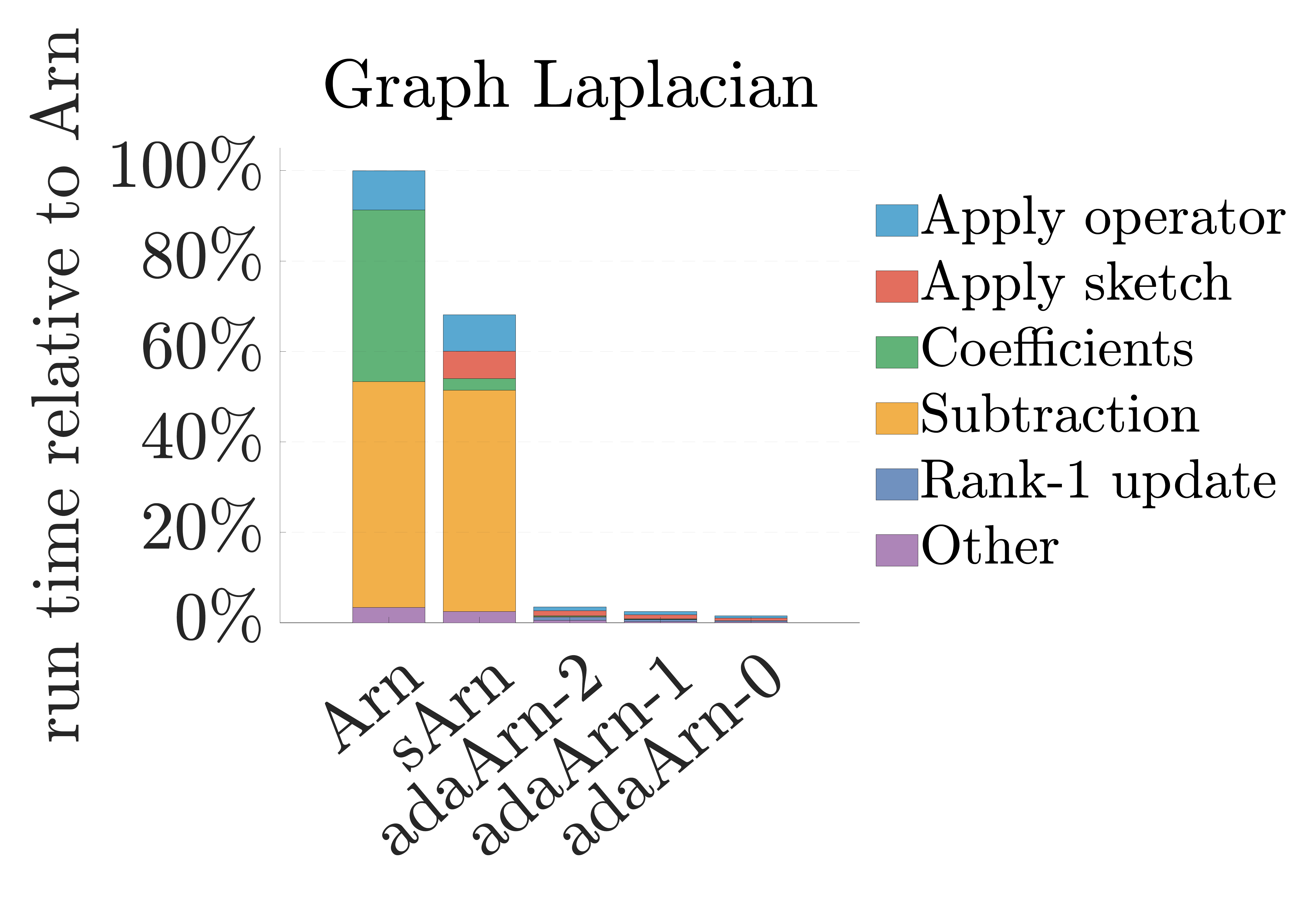}
\caption{Runtime breakdown for the fractional graph Laplacian example
without thick restarting.
The left panel shows the total runtime, normalized by that of Arn.
The right panel shows the average basis-generation time
per restart cycle, normalized by the corresponding average for Arn.
Fixed-dimension methods use \(m=150\) and \(s=4m=600\);
adaptive methods use \(m_{\max}=150\) and \(\eta=4\).}
\label{fig:graph-laplacian-invsqrt-sfactor-4-ratios}
\end{figure}

\begin{table}[htbp]
\centering
\caption{Adaptive Arnoldi methods with thick restarting for the graph Laplacian inverse square root example (\(m_{\max}=150\), \(\eta=4\)). Columns report the relative error, runtime in seconds, number of restart cycles \(n_{\mathrm{iter}}\), number of matrix-vector products \(n_{\mathrm{matvec}}\), and mean actual Krylov subspace dimension per cycle \(m_{\mathrm{act}}\), including retained directions. Speedup is the percentage runtime reduction relative to Arn without thick restarting (\(m=150\)), \((t_{\mathrm{Arn}}-t)/t_{\mathrm{Arn}}\times100\%\).}
\label{tab:graph-laplacian-invsqrt-thick-restarting-krylov-dim-150-sfactor-4}
\resizebox{\textwidth}{!}{%
\begin{tabular}{@{}lrrrrrr@{}}
\toprule
method & error & time & \(n_{\mathrm{iter}}\) & \(n_{\mathrm{matvec}}\) & \(m_{\mathrm{act}}\) & speedup \\
\midrule
adaArn-2 (with thick restarting) & 3.0e-8 & 3.3 & 28 & 418 & 19.8 & \(73.4\%\) \\
adaArn-1 (with thick restarting) & 1.2e-8 & 3.4 & 37 & 481 & 17.9 & \(72.9\%\) \\
adaArn-0 (with thick restarting) & 4.2e-9 & 3.2 & 60 & 539 & 13.9 & \(74.6\%\) \\
\bottomrule
\end{tabular}%
}
\end{table}

Without thick restarting, sArn and sHmArn are the fastest methods
in their respective groups, with runtimes of \(8.9\) and \(9.2\) seconds,
representing reductions of \(28.6\%\) and \(30.0\%\)
relative to Arn and HmArn
(Table~\ref{tab:graph-laplacian-invsqrt-sfactor-4}).
Basis generation still accounts for almost all the runtime
of the fixed-dimension methods
(Figure~\ref{fig:graph-laplacian-invsqrt-sfactor-4-ratios}).
Here, applying \(\Mat{A}\) is inexpensive,
so reducing orthogonalization costs through sketching
provides substantial savings.
Only the Arnoldi variants are plotted,
since the harmonic variants exhibit similar convergence
and runtime breakdowns without thick restarting.

Without thick restarting, the adaptive methods require
\(299\)--\(638\) restart cycles.
Quadrature uses \(128\)--\(256\) nodes in most cycles
after the initial stages and consequently accounts for
a substantial fraction of the total runtime.
The accumulated costs of basis generation and quadrature
outweigh the savings per cycle,
making these methods slower than the fixed-dimension methods.

Thick restarting substantially improves adaptive convergence,
as shown in
Figure~\ref{fig:graph-laplacian-invsqrt-sfactor-4-err-thick}
and
Table~\ref{tab:graph-laplacian-invsqrt-thick-restarting-krylov-dim-150-sfactor-4}.
For the adaptive Arnoldi methods,
it reduces the number of cycles from \(325\)--\(638\) to \(28\)--\(60\),
the number of matrix-vector products from \(4743\)--\(5740\)
to \(418\)--\(539\),
and the runtime from \(36.0\)--\(52.8\) to \(3.2\)--\(3.4\) seconds.
These methods are approximately
\(3.7\)--\(3.9\) times as fast as Arn.

\subsection{Convection-diffusion example}
\label{subsec:conv_diff}

In this example, \(\Mat{A}\) is obtained by discretizing a two-dimensional convection-diffusion operator on the unit square with a constant convection field in the direction \([1,-1]\) and diffusion coefficient \(\nu = 1\).
More precisely,
\begin{equation*}
\Mat{A}
=
\frac{\nu}{h^{2}}
\bigl(
\Mat{L} \otimes \Mat{I}
+
\Mat{I} \otimes \Mat{L}
\bigr)
+
\frac{1}{h}
\bigl(
\Mat{C} \otimes \Mat{I}
+
\Mat{I} \otimes \Mat{C}^{\trans}
\bigr)
\in
\real^{n^{2} \times n^{2}},
\end{equation*}
where \(h = 1/(n + 1)\), 
\(\Mat{L} \in \real^{n \times n}\) is the tridiagonal matrix 
with diagonal entries \(2\) 
and sub- and superdiagonal entries \(-1\), 
and \(\Mat{C} \in \real^{n \times n}\) is the matrix 
with diagonal entries \(1\) and subdiagonal entries \(-1\).
We aim to compute \(f_{2}(\Mat{A})\Mat{b}\).
In our experiments, \(n\) is set to be \(512\) so that \(N = n^{2} = 512^{2} = 262144\),
and \(\Mat{b} \in \complex^{N}\) is a complex vector
whose real and imaginary parts are drawn from the standard normal distribution.
For the fixed-dimension methods, we set \(m \in \{150, 200\}\) 
while for the adaptive methods, we set \(m_{\max} = 200\).

\begin{table}[htbp]
\centering
\caption{Relative error, runtime in seconds, number of restart cycles \(n_{\mathrm{iter}}\), number of matrix-vector products \(n_{\mathrm{matvec}}\), and average actual Krylov subspace dimension \(m_{\mathrm{act}}\) per restart cycle for the convection-diffusion logarithm example. The column ``speedup'' reports the percentage runtime reduction \((t_{\mathrm{base}}-t_{\mathrm{curr}})/t_{\mathrm{base}}\times100\%\), where the baseline is Arn (left group) or HmArn (right group), using the fixed-dimension baseline at \(m=m_{\max}\) for the adaptive section. For fixed-dimension methods, \(m\) is given in the section heading and \(s=4m\). Adaptive methods use \(m_{\max}=200\) and \(\eta=4\), and appear in a separate section.}
\label{tab:conv-diff-log-sfactor-4}
\resizebox{\textwidth}{!}{%
\begin{tabular}{@{}lrrrrrr@{\quad}|@{\quad}lrrrrrr@{}}
\toprule
method & error & time & \(n_{\mathrm{iter}}\) & \(n_{\mathrm{matvec}}\) & \(m_{\mathrm{act}}\) & speedup & method & error & time & \(n_{\mathrm{iter}}\) & \(n_{\mathrm{matvec}}\) & \(m_{\mathrm{act}}\) & speedup \\
\midrule
\multicolumn{14}{c}{\(m = 150\), fixed} \\
\midrule
Arn & 1.9e-9 & 113.4 & 15 & 2250 & 150.0 & \(0.0\%\) & HmArn & 5.1e-10 & 127.5 & 16 & 2400 & 150.0 & \(0.0\%\) \\
sArn & 1.6e-9 & 86.2 & 19 & 2850 & 150.0 & \(23.9\%\) & sHmArn & 1.8e-8 & 77.8 & 16 & 2400 & 150.0 & \(39.0\%\) \\
\midrule
\multicolumn{14}{c}{\(m = 200\), fixed} \\
\midrule
Arn & 1.3e-10 & 144.3 & 11 & 2200 & 200.0 & \(0.0\%\) & HmArn & 5.9e-11 & 177.3 & 13 & 2600 & 200.0 & \(0.0\%\) \\
sArn & 5.0e-10 & 101.7 & 14 & 2800 & 200.0 & \(29.5\%\) & sHmArn & 1.7e-8 & 97.4 & 13 & 2600 & 200.0 & \(45.1\%\) \\
\midrule
\multicolumn{14}{c}{\(m_{\max} = 200\), adaptive} \\
\midrule
adaArn-2 & 1.7e-10 & 84.1 & 11 & 2172 & 197.5 & \(41.8\%\) & adaHmArn-2 & 2.7e-10 & 98.6 & 12 & 2389 & 199.1 & \(44.4\%\) \\
adaArn-1 & 3.6e-7 & 184.4 & 594 & 10102 & 17.0 & \(-27.8\%\) & adaHmArn-1 & 6.1e-7 & 232.5 & 414 & 7855 & 19.0 & \(-31.1\%\) \\
adaArn-0 & 1.7e-6 & 320.7 & 2625 & 21000 & 8.0 & \(-122.2\%\) & adaHmArn-0 & 3.2e-6 & 352.7 & 2275 & 18200 & 8.0 & \(-98.9\%\) \\
\bottomrule
\end{tabular}%
}
\end{table}

\begin{figure}[tbhp]
\centering
\includegraphics[width=.49\textwidth]
{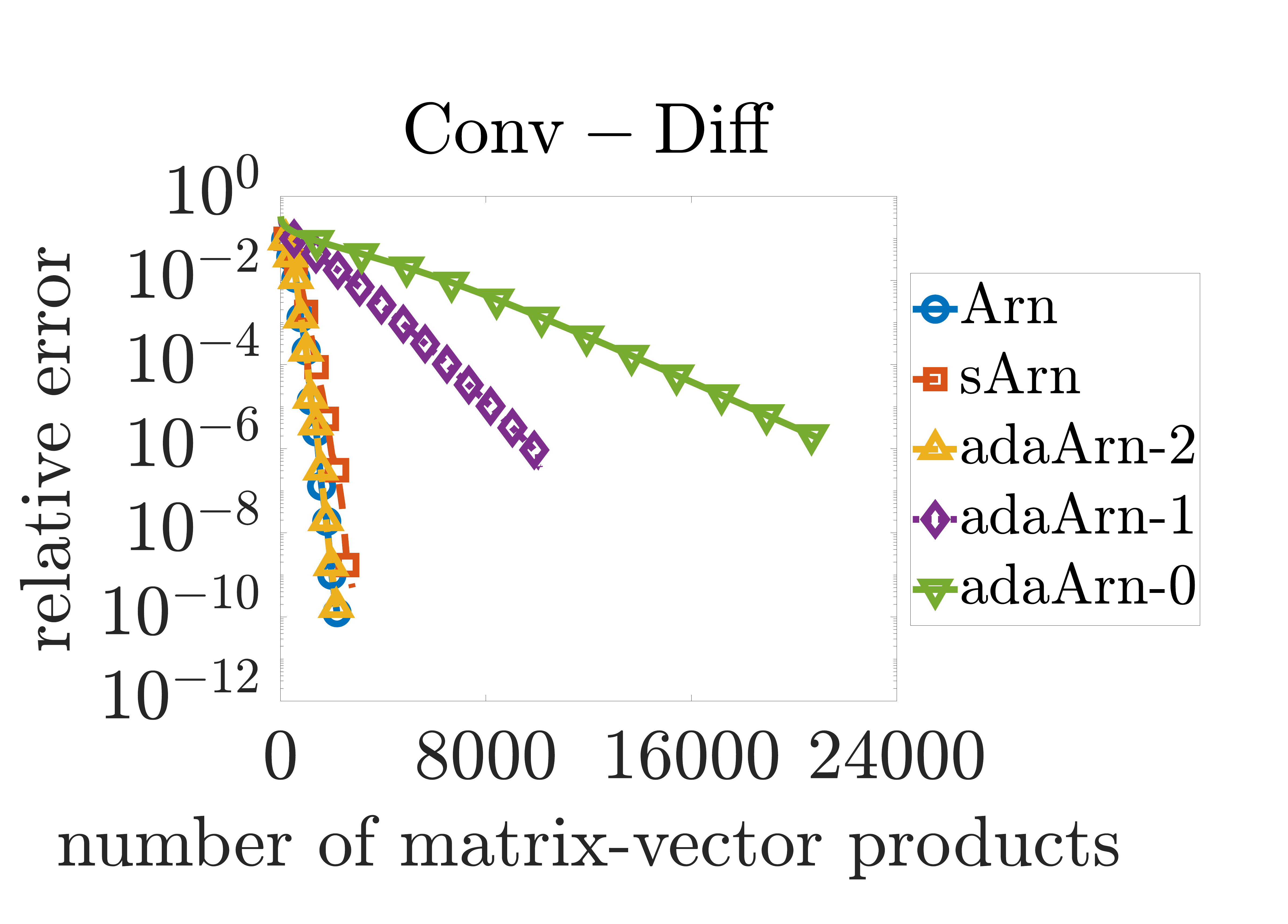}
\includegraphics[width=.49\textwidth]
{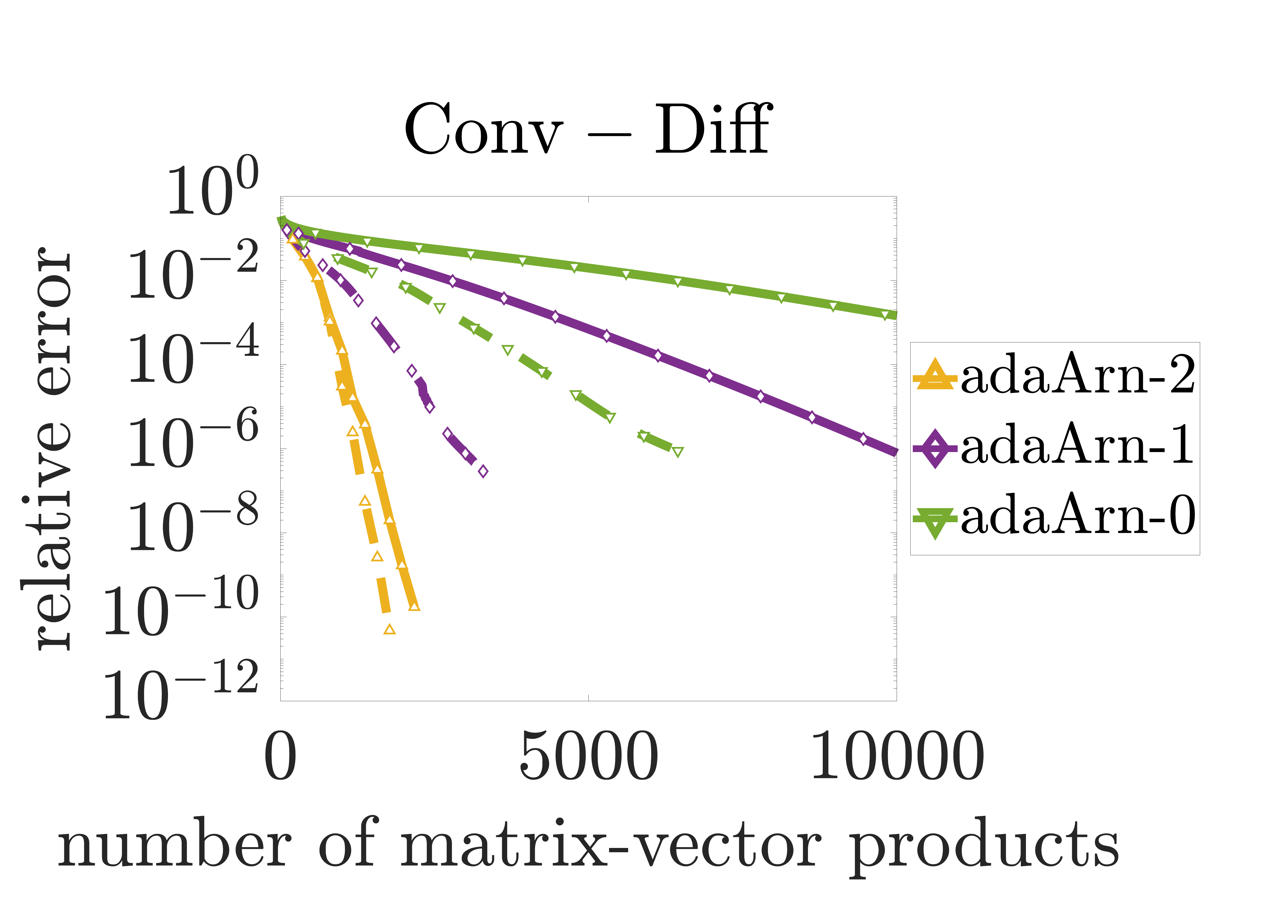}
\caption{Convergence for the convection-diffusion example.
Left: relative error versus the number of matrix-vector products
without thick restarting.
Right: adaptive Arnoldi methods without thick restarting
(solid curves) and with thick restarting (dashed curves),
shown over a smaller range of matrix-vector products.
Fixed-dimension methods use \(m=200\) and \(s=4m=800\);
adaptive methods use \(m_{\max}=200\) and \(\eta=4\).
The target retention dimension for thick restarting is \(5\),
increased when necessary to preserve a conjugate pair.}
\label{fig:conv-diff-log-sfactor-4-err-thick}
\end{figure}

\begin{figure}[tbhp]
\centering
\includegraphics[width=.49\textwidth]
{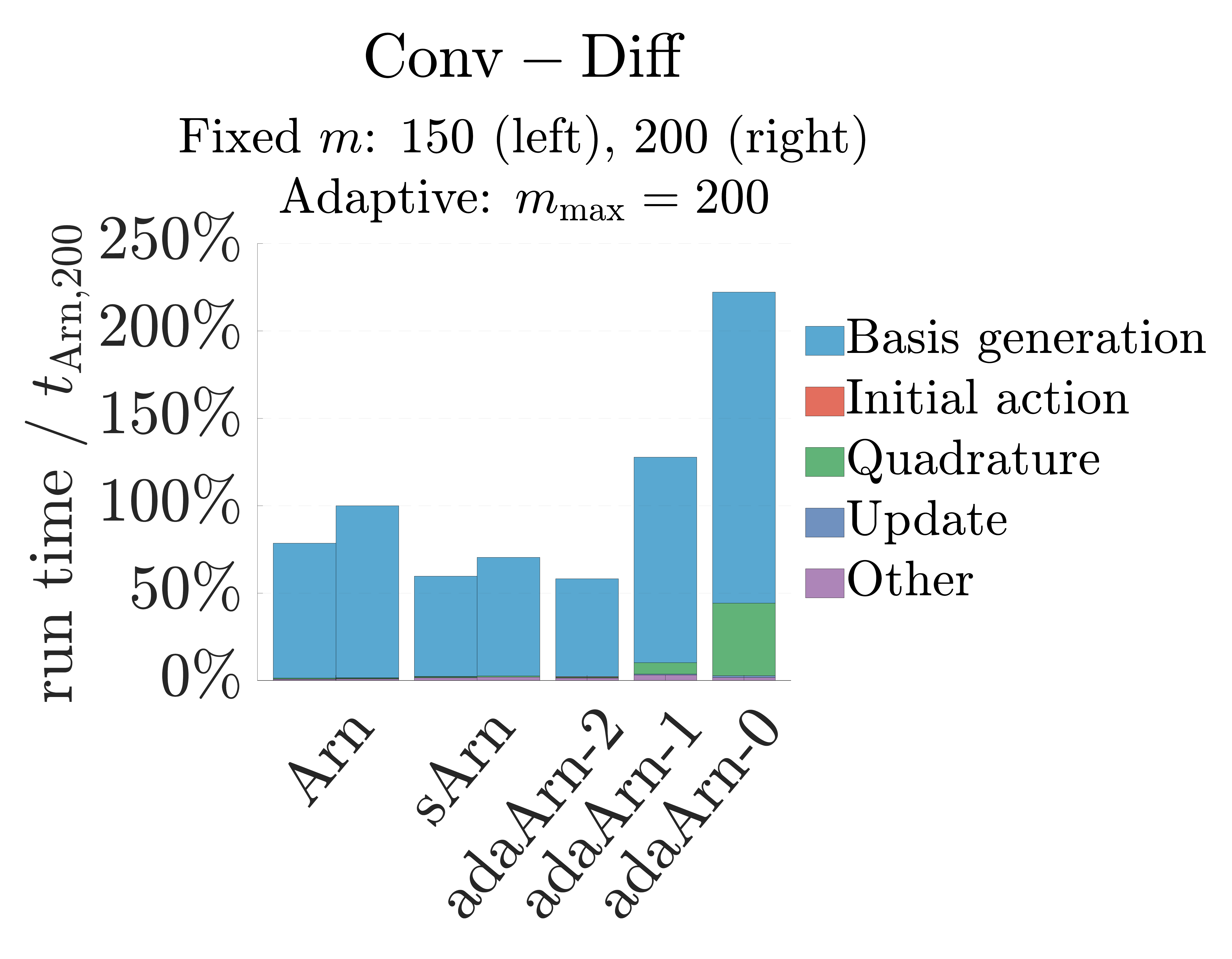}
\includegraphics[width=.49\textwidth]
{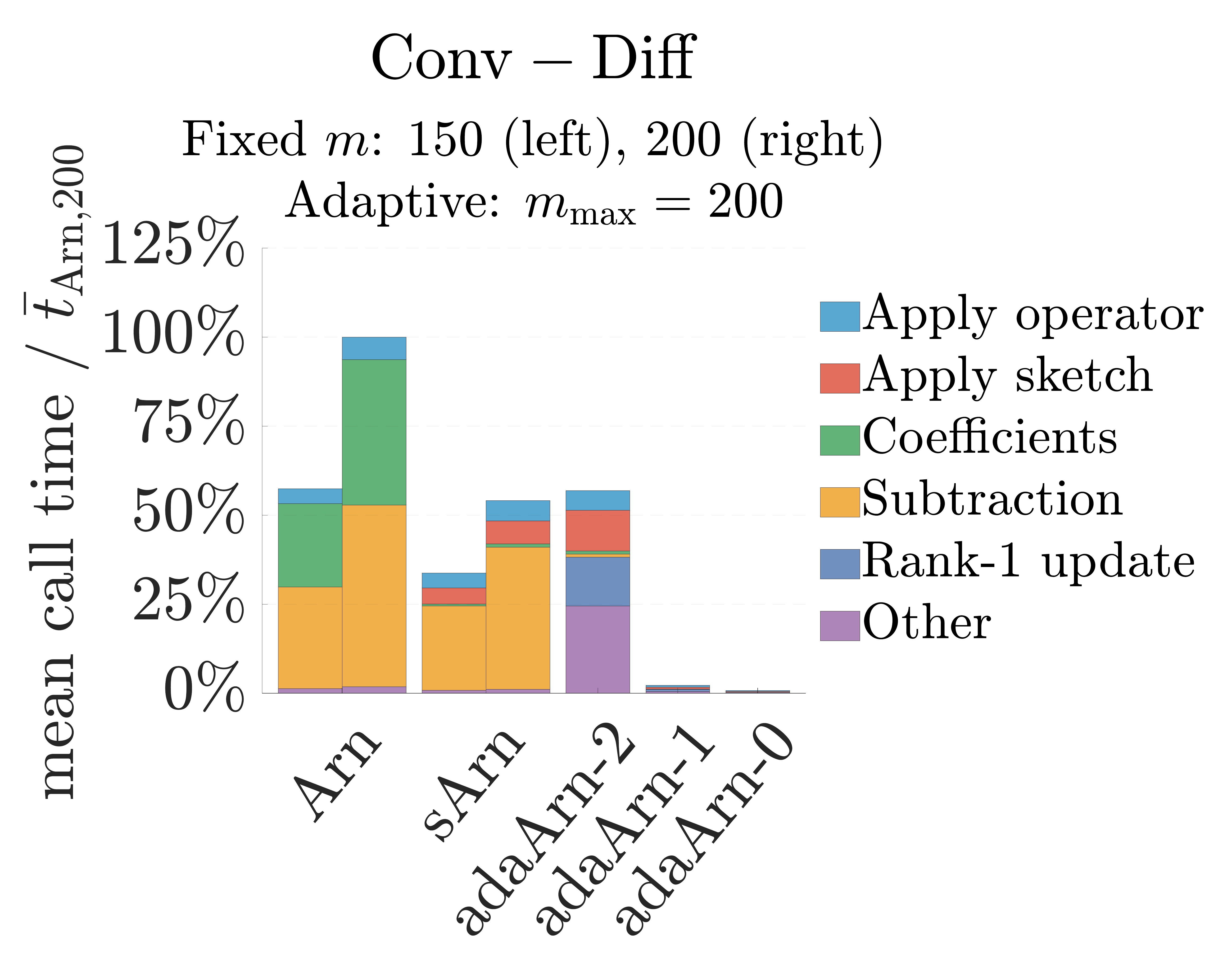}
\caption{Runtime breakdown for the convection-diffusion example
without thick restarting.
Left: total runtime, normalized by that of Arn with \(m=200\).
Right: average basis-generation time per restart cycle,
normalized by the corresponding average for Arn with \(m=200\).
For each fixed-dimension method, the left and right bars
correspond to \(m=150\) and \(m=200\), respectively,
with \(s=4m\).
Adaptive methods use \(m_{\max}=200\) and \(\eta=4\).}
\label{fig:conv-diff-log-sfactor-4-ratios}
\end{figure}

\begin{table}[htbp]
\centering
\caption{Adaptive Arnoldi methods with thick restarting for the convection-diffusion logarithm example (\(m_{\max}=200\), \(\eta=4\)). Columns report the relative error, runtime in seconds, number of restart cycles \(n_{\mathrm{iter}}\), number of matrix-vector products \(n_{\mathrm{matvec}}\), and mean actual Krylov subspace dimension per cycle \(m_{\mathrm{act}}\), including retained directions. Speedup is the percentage runtime reduction relative to Arn without thick restarting (\(m=200\)), \((t_{\mathrm{Arn}}-t)/t_{\mathrm{Arn}}\times100\%\).}
\label{tab:conv-diff-log-thick-restarting-krylov-dim-200-sfactor-4}
\resizebox{\textwidth}{!}{%
\begin{tabular}{@{}lrrrrrr@{}}
\toprule
method & error & time & \(n_{\mathrm{iter}}\) & \(n_{\mathrm{matvec}}\) & \(m_{\mathrm{act}}\) & speedup \\
\midrule
adaArn-2 (with thick restarting) & 4.7e-11 & 77.3 & 9 & 1771 & 201.2 & \(46.5\%\) \\
adaArn-1 (with thick restarting) & 1.6e-7 & 80.5 & 203 & 3459 & 22.0 & \(44.2\%\) \\
adaArn-0 (with thick restarting) & 7.7e-7 & 142.4 & 819 & 6552 & 13.0 & \(1.3\%\) \\
\bottomrule
\end{tabular}%
}
\end{table}

Table~\ref{tab:conv-diff-log-sfactor-4} shows that
sketching substantially reduces the runtime.
For \(m=150\) and \(m=200\), sArn reduces the total runtime
relative to Arn by \(23.9\%\) and \(29.5\%\), respectively,
while sHmArn achieves reductions of \(39.0\%\) and \(45.1\%\)
relative to HmArn.
Here, applying \(\Mat{A}\) is inexpensive,
so orthogonalization accounts for a large fraction
of the basis-generation cost.
Sketching reduces this cost substantially,
as illustrated in
Figure~\ref{fig:conv-diff-log-sfactor-4-ratios}.

The average basis-generation time per cycle
is reduced by \(41.2\%\) for \(m=150\)
and \(45.9\%\) for \(m=200\).
These savings do not translate directly into
comparable reductions in total runtime,
because sArn requires more restart cycles than Arn:
\(19\) rather than \(15\) for \(m=150\),
and \(14\) rather than \(11\) for \(m=200\).
For example, at \(m=200\), the ratio of the accumulated
basis-generation times is approximately
\[
\frac{T_{\mathrm{basis,sArn}}}{T_{\mathrm{basis,Arn}}}
=
\frac{6.99}{12.92}\,\frac{14}{11}
\approx 0.689.
\]
Thus, the additional cycles offset part of the savings per cycle,
leaving a reduction of approximately \(31.1\%\)
in the accumulated basis-generation time.
Including the remaining costs gives
the observed total runtime reduction of \(29.5\%\).

Increasing \(m\) from \(150\) to \(200\)
reduces the number of restart cycles
for all four fixed-dimension methods,
but increases their total runtimes.
For Arn and sArn, the number of matrix-vector products
decreases only slightly,
from \(2250\) to \(2200\) and from \(2850\) to \(2800\),
respectively.
For HmArn and sHmArn, it increases from \(2400\) to \(2600\).
The larger subspaces therefore make each cycle more expensive
without providing sufficient improvement in convergence
to compensate.
Although the relative savings from sketching are greater
at \(m=200\), all four fixed-dimension methods
are faster at \(m=150\).

Among the methods with \(m=m_{\max}=200\)
without thick restarting, adaArn-2 is the fastest,
requiring \(84.1\) seconds and \(2172\) matrix-vector products.
It also uses the fewest matrix-vector products
among all configurations in
Table~\ref{tab:conv-diff-log-sfactor-4},
although sHmArn with \(m=150\) has a shorter runtime
of \(77.8\) seconds.
In contrast, adaArn-1 and adaArn-0 generate much smaller
subspaces but require many more restart cycles,
making them slower than Arn.

Thick restarting substantially improves the convergence
of adaArn-1 and adaArn-0;
see Figure~\ref{fig:conv-diff-log-sfactor-4-err-thick}
and
Table~\ref{tab:conv-diff-log-thick-restarting-krylov-dim-200-sfactor-4}.
Their matrix-vector counts decrease
from \(10102\) to \(3459\) and from \(21000\) to \(6552\),
while their runtimes decrease
from \(184.4\) to \(80.5\) seconds
and from \(320.7\) to \(142.4\) seconds, respectively.
Nevertheless, adaArn-0 remains only \(1.3\%\) faster
than Arn with \(m=200\).
The benefit is more modest for adaArn-2,
whose runtime decreases from \(84.1\) to \(77.3\) seconds
and whose matrix-vector count decreases from \(2172\) to \(1771\).
With thick restarting, adaArn-2 is the fastest
among all configurations reported here
and requires the fewest matrix-vector products.

\section{Conclusions}
\label{sec:conclusions}

We have developed a sketch-and-restart framework for computing
\(f(\Mat{A})\Mat{b}\) with large sparse non-Hermitian matrices.
The framework combines quadrature-based restarting with
Arnoldi-like decompositions generated by sketched or truncated
Krylov processes.
The key observation is that several Krylov and sketched Krylov
approximations share the common form
\(\Basis_{m}f(\Hessen_{m})\Mat{e}_{1}\beta\),
which leads to a unified error representation and restarting mechanism.
Within this framework, we combined the existing s-Arnoldi process
with quadrature-based restarting and introduced a new
sHm-Arnoldi process together with its restarted variant.
We further developed adaptive restarted Arnoldi and HmArnoldi
methods based on the \(t\)-Arnoldi process and a rank-\(1\) update.
For Stieltjes functions of positive real matrices, we derived
conditional convergence bounds for the restarted sHmArnoldi method
under suitable embedding and sketched harmonic Ritz-value assumptions,
and provided a sufficient condition ensuring that these Ritz values
lie in the open right half-plane.
More broadly, the proposed framework separates the restarting
mechanism from the basis generation procedure, making it possible
to incorporate other randomized Krylov processes within the same
framework.

Numerical experiments show that sketching reduces orthogonalization
costs and total runtime in the reported examples.
The adaptive methods select the restart length automatically
and can substantially reduce basis storage, although the smaller
subspaces may require more restart cycles and increase the total runtime.
For the adaptive Arnoldi variants, thick restarting substantially
improves convergence in the graph Laplacian and convection--diffusion
examples, while retaining small subspaces under aggressive truncation.
The use of BLAS-2 operations in the final rank-\(1\) update
also motivates further investigation of high-performance implementations.

\section*{Acknowledgments}

JL thanks the China Scholarship Council~(CSC) for supporting his visit to The University of Manchester. 
SG and LN acknowledge funding from the UK's Engineering and Physical Sciences Research Council~(EPSRC grant EP/Z533786/1).
SG is supported by Royal Society Industry Fellowship IF/R1/231032. 

\bibliographystyle{siamplain}
\bibliography{ref}

\end{document}